\documentclass[12pt,a4paper]{article}
\usepackage{amsfonts}
\usepackage{amssymb}
\usepackage{amsthm}
\usepackage[english]{babel}
\usepackage{hhline}
\usepackage[ansinew]{inputenc}
\usepackage{amsmath}
\usepackage[all,2cell,ps]{xy}
\usepackage{qsymbols}
\usepackage{color}
\usepackage{epsfig}
\usepackage{color}
\usepackage{graphics}
\usepackage{graphicx}%
\usepackage{enumerate}

\usepackage{amssymb}

\theoremstyle{plain}
\newtheorem{thm}{Theorem}
\newtheorem{prop}{Proposition}
\newtheorem{cor}[prop]{Corollary}
\newtheorem{lem}[prop]{Lemma}
\newtheorem{defn}{Definition}
\newtheorem{rem}{Remark}
\newtheorem{ex}{Example}
\newcommand{\ra}{\rightarrow}

\newcommand{\we}{\wedge}
\newcommand{\X}{\mathbf{X}}
\newcommand{\D}{\mathbf{D}}
\newcommand{\KA}{\mbox{{\Large $\kappa$}}}
\newcommand{\Ra}{\Rightarrow}
\newcommand{\f}{(\ )^*}
\newcommand{\edge}[1]{\ar@{-}[#1]}

\newcommand{\clx}{\overline{\{x\}}}
\newcommand{\cly}{\overline{\{y\}}}
\newcommand{\sat}{\mathrm{sat}}
\newcommand{\pd}{\leq_{d}}
\newcommand{\Fil}{\mathsf{Fil}}
\newcommand{\va}{\mathsf{V}}
\newcommand{\h}{\mathsf{H}}
\newcommand{\s}{\mathsf{S}}
\newcommand{\p}{\mathsf{P}}
\newcommand{\is}{\mathsf{I}}

\newcommand{\K}{\mathrm{K}}

\newcommand{\U}{\mathrm{U}}
\newcommand{\Sg}{\mathrm{Sg}}
\newcommand{\Hey}{\mathsf{Hey}}
\newcommand{\PHey}{\mathsf{PHey}}
\newcommand{\Hil}{\mathsf{Hil}}
\newcommand{\Hilz}{\mathsf{Hil_0}}
\newcommand{\PHil}{\mathsf{PHil}}
\newcommand{\PHilz}{\mathsf{PHil_0}}
\newcommand{\fPHilz}{\mathsf{fPHil_0}}
\newcommand{\HS}{\mathsf{HS}}
\newcommand{\PHS}{\mathsf{PHS}}
\newcommand{\PHSz}{\mathsf{PHS_0}}
\newcommand{\IS}{\mathsf{IS}}

\newcommand{\Esa}{\mathsf{Es}}

\newcommand{\FC}{\mathrm{S}}

\newcommand{\G}{\mathrm{G}}

\newcommand{\cPHey}{\coprod_{\PHey}}

\newcommand{\cPHilz}{\coprod_{\PHilz}}
\newcommand{\cfPHilz}{\coprod_{f}}
\newcommand{\Pos}{\mathrm{Pos}}
\newcommand{\Esap}{\mathrm{PEs}}
\newcommand{\Ffin}{\mathrm{F_{fin}}}
\newcommand{\hFor}{\mathrm{hFor}}
\newcommand{\ChFor}{\mathrm{ChFor}}
\newcommand{\US}{\mathrm{US}}

\begin{document}

\title{Prelinear Hilbert algebras}

\author{Castiglioni Jos\'e Luis, Celani Sergio A. and San Mart\'{\i}n Hern\'an J.}

\maketitle

\begin{abstract}
In this paper we give an explicit description of the left adjoint
of the forgetful functor from the algebraic category of G\"odel algebras (i. e., prelinear
Heyting algebras) to the algebraic category of bounded prelinear
Hilbert algebras. We apply this result in order to study possible
descriptions of the coproduct of two finite algebras in the
algebraic category of prelinear Hilbert algebras.
\end{abstract}

\section{Introduction and basic results}\label{s1}

We assume that the reader is familiar with the theory of Heyting
algebras \cite{BD}, which is the algebraic counterpart of
Intuitionistic Propositional Logic. Hilbert algebras represent the
algebraic counterpart of the implicative fragment of
Intuitionistic Propositional Logic, and they were introduced in
early 50's by Henkin for some investigations of implication in
intuicionistic and other non-classical logics (\cite{R}, pp. 16).
In the 60's, Hilbert algebras were studied especially by Horn
\cite{H} and Diego \cite{D}.

\begin{defn}
A \emph{Hilbert algebra} is an algebra $(H,\ra,1)$ of type $(2,0)$
which satisfies the following conditions for every $a,b,c\in H$:
\begin{enumerate}[\normalfont a)]
\item $a\ra (b\ra a) = 1$, \item $(a\ra (b\ra c)) \ra ((a\ra b)\ra
(a\ra c)) = 1$, \item if $a\ra b = b\ra a = 1$ then $a = b$.
\end{enumerate}
\end{defn}

It is known that the class of Hilbert algebras is a variety. In
every Hilbert algebra we have the partial order $\leq$ given by
\[
a\leq b\;\text{if and only if}\; a\ra b = 1,
\]
which is called \emph{natural order}. Relative to the natural
order on $H$ we have that $1$ is the greatest element.

We say that an algebra $(H,\ra,1,0)$ of type $(2,0,0)$ is a
\emph{bounded} Hilbert algebra if $(H,1)$ is a Hilbert algebra and
$0\leq a$ for every $a\in H$.

The following lemma is part of the folklore of Hilbert algebras
(see for example \cite{B}).

\begin{lem}
Let $H$ be a Hilbert algebra and $a,b,c\in H$. Then the following
conditions are satisfied:
\begin{enumerate}[\normalfont a)]
\item $a\ra a = 1$, \item $1\ra a = a$, \item $a\ra (b\ra c) =
b\ra (a \ra c)$, \item $a\ra (b\ra c) = (a\ra b)\ra (a\ra c)$,
\item if $a\leq b$ then $c\ra a \leq c\ra b$ and $b\ra c \leq a\ra
c$.
\end{enumerate}
\end{lem}

Some additional properties of Hilbert algebras can be found in
\cite{B,D}.

For the general development of Hilbert algebras, the notion of
implicative filter plays an important role. Let $H$ be a Hilbert
algebra. A subset $F\subseteq H$ is said to be an
\emph{implicative filter} if $1\in F$ and $b\in F$ whenever $a\in
F$ and $a\ra b \in F$. We will denote the set of all implicative
filters of $H$ by $\Fil(H)$. Recall that if $X \subseteq H$, we
define the filter generated by $X$ as the least filter of $H$ that
contains the set $X$, which will be denoted by $F(X)$. There is an
explicit description for $F(X)$ (see \cite[Lemma 2.3]{Bu}). More
precisely, if $X\neq \emptyset$ we have that
\[
F(X) =\{x \in H: a_1\ra (a_2\ra \cdots (a_n \ra x) \ldots)=1
\;\text{for some}\; a_{1}, \ldots,a_n\in X\}
\]
and $F(\emptyset) = \{1\}$. Let us consider a poset $(X,\leq)$.
For each $Y \subseteq X$, \emph{the increasing set generated by
$Y$} is defined by $[Y) = \{x\in X:\;\text{there is}\;y\in
Y\:\text{such that}\;y\leq x\}$. The \emph{decreasing set
generated by $Y$} is dually defined. If $Y = \{y\}$, then we will
write $[y)$ and $(y]$ instead of $[\{y\})$ and $(\{y\}]$,
respectively. We say that $Y$ is an \emph{upset} if $Y = [Y)$, and
a downset if $Y = (Y]$. If $H$ is a Hilbert algebra then every
filter of $H$ is an upset and for every $a\in H$ we have that
$F(\{a\}) = [a)$.

Let $H$ be a Hilbert algebra and $F\in \Fil(H)$. We say that $F$
is \emph{irreducible} if $F$ is proper (i.e., $F\neq H$) and for
any implicative filters $F_1, F_2$ such that $F = F_1 \cap F_2$ we
have that $F = F_1$ or $F = F_2$. We write $X(H)$ for the set of
irreducible implicative filters of $H$. We also write $X(H)$ for
the poset of irreducible implicative filters of $H$ where the
order is given by the inclusion.

The proof of the following lemma can be found in \cite{D}.

\begin{lem} \label{irred}
Let $H$ be a Hilbert algebra and $F\in\mathrm{Fi}(H)$. The
following statements are equivalent:
\begin{enumerate}[\normalfont 1.]
\item $F\in X(H)$. \item For every $a,b\in H$ such that $a,b\notin
F$ there exists $c\notin F$ such that $a,b\leq c$. \item For every
$a,b\in H$ such that $a,b\notin F$ there exists $c\notin F$ such
that $a\rightarrow c,b\rightarrow c\in F$.
\end{enumerate}
\end{lem}

Let $H$ be a Hilbert algebra. An \emph{order ideal} of $H$ is a
downset $I$ of $H$ such that for each $a,b\in I$, there exists
$c\in I$ such that $a\leq c$ and $b\leq c$.

The following is \cite[Theorem 2.6]{Cel}.

\begin{thm} \label{tfp}
Let $H$ be a Hilbert algebra, $F\in\mathrm{Fi}(H)$ and $I$ an
order ideal of $H$ such that $F\cap I=\emptyset$. Then there
exists $P\in X(H)$ such that $F\subseteq P$ and $P\cap
I=\emptyset$.
\end{thm}

The following corollaries are known in the literature, and can be
obtained by using the previous theorem.

\begin{cor}\label{tfpc2}
Let $H$ be a Hilbert algebra and $a,b\in H$ such that $a\nleq b$.
Then there exists $P\in X(H)$ such that $a\in P$ and $b\notin P$.
\end{cor}

\begin{cor} \label{tfpc3}
Let $H$ be a Hilbert algebra, $F\in \Fil(H)$ and $a,b\in H$. Then
$a\ra b \notin F$ if and only if there is $P\in X(H)$ such that
$F\subseteq P$, $a\in P$ and $b\notin P$.
\end{cor}

We give a table with some of the categories we shall consider in
this paper: \vspace{10pt}

\small

\begin{tabular}{c|c| c c}

$\mathbf{Category}$ &       $\mathbf{Objects}$                                                     & $\mathbf{Morphisms}$ \\
\hline
&&\\
$\Hey$                      &Heyting algebras                                                      & Algebra homomorphisms   \\
$\PHey$                     &Prelinear Heyting algebras                                            & Algebra homomorphisms   \\
$\Hil$                      &Hilbert algebras                                                      & Algebra homomorphisms       \\
$\Hilz$                     &Bounded Hilbert algebras                                              & Algebra homomorphisms\\
$\PHil$                     &Prelinear Hilbert algebras                                            & Algebra homomorphisms\\
$\PHilz$                    &Bounded Prelinear Hilbert algebras                                    & Algebra homomorphisms \\
$\fPHilz$                   &Bounded finite prelinear Hilbert algebras                             & Algebra homomorphisms \\
$\IS$                       &Implicative semilattices                                              & Algebra homomorphisms \\
$\Pos$                      &Posets                                                                & Order preserving maps \\
$\Esap$                     &Esakia spaces which are root systems                                  & Continuous p-morphisms \\
$\HS$                       &Hilbert spaces                                                        & Certain continuous maps \\
$\PHS$                      &Hilbert spaces which are root systems                                 & Certain morphisms of $\HS$  \\
$\PHSz$                     &Certain objects of $\PHS$                                             & Certain morphisms of $\PHS$ \\
$\Ffin$                     &Finite forests                                                        & Open maps \\
$\ChFor$                    &Finite forests with a family of its subsets                           & Certain binary relations \\
$\hFor$                     &Finite forests with a family of its subsets                           & Certain open maps \\

\end{tabular}

\vspace{10pt}

The paper is organized as follows. In Section \ref{s2} we study
some properties of the variety of prelinear Hilbert algebras,
which was considered by Monteiro in \cite{A. Monteiro 1996}. In
Section \ref{s3} we present a categorical equivalence for the
algebraic category of bounded prelinear Hilbert algebras. We use
it in order to give an explicit description of the left adjoint of
the forgetful functor from the algebraic category of prelinear
Heyting algebras to the algebraic category of bounded prelinear
Hilbert algebras. The ideas used in this section are similar to
that developed in \cite{CSM}. In Section \ref{s4} we apply results
of the previous section in order to study some descriptions of the
coproduct of two finite algebras in the algebraic category of
bounded prelinear Hilbert algebras. In Section \ref{s5} we present
a description of the product in some category of finite forests.
In Section \ref{s6} we use results of Section \ref{s5} in order to
give an explicit description of the product of two objects in
certain category of finite forests $F$ endowed with a
distinguished family of subsets of $F$, which is equivalent to the
category of finite bounded prelinear Hilbert algebras. This
property allow us to obtain an explicit description of the
coproduct of two finite algebras in the algebraic category of
bounded prelinear Hilbert algebras. Finally, in Sectión \ref{s7}
we give an explicit description of the coproduct of two finite
algebras in the algebraic category of prelinear Hilbert algebras
in terms of the coproduct in the algebraic category of bounded
prelinear Hilbert algebras.

\section{Prelinear Hilbert algebras}\label{s2}

In this section we give some properties of prelinear Hilbert
algebras, which were introduced and studied by Monteiro in
\cite{A. Monteiro 1996}.

\begin{defn}
We say that a Hilbert algebra is a Hilbert chain if its natural
order is total. We call prelinear Hilbert algebra to the members
of the variety generated by the class of Hilbert chains.
\end{defn}

Let $H$ be a Hilbert algebra. For every $a,b,c \in H$ we define
\[
l(a,b,c) := ((a\ra b)\ra c)\ra(((b\ra a)\ra c)\ra c.
\]
Considerer the following equation:
\begin{equation*}
\mathrm{(P)} \;\text{}\;\;\;\; l(a,b,c)=1.
\end{equation*}

Let $\K$ be a class of algebras of the same type. As usual, we
write $\va(\K)$ for the variety generated by $\K$, $\is(K)$ for
the isomorphic members of $\K$, $\h(\K)$ for the homomorphic image
of members of $\K$, $\s(\K)$ for the subalgebras of members of
$\K$ and $\p(\K)$ for the direct product of members of $\K$.
Recall that by Tarski's theorem we have that $\h\s\p(\K) =
\va(\K)$ (see \cite{BD}).

The following is \cite[Theorem 2.2] {A. Monteiro 1996}. By clarity
in the exposition of the present paper we think it is convenient
to present a sketch of the proof.

\begin{prop} \label{mp1}
Let $\K$ be the class of Hilbert chains. Then the variety
$\va(\K)$ is characterized by an equational basis for the variety
of Hilbert algebras with the additional equation $\mathrm{(P)}$.
\end{prop}

\begin{proof}
Since every Hilbert chain satisfies the equation $\mathrm{(P)}$
then every member of $\va(\K)$ satisfies $\mathrm{(P)}$.
Conversely, let $H$ be a Hilbert algebra which satisfies
$\mathrm{(P)}$. If $H$ is trivial then $H\in \va(\K)$. Assume that
$H$ is not trivial. Then it follows from \cite[Theorem 7.1]{A.
Monteiro 1996} that $H \in \is\s\p(\K)$. Since $\is\s\p(\K)
\subseteq \h\s\p(\K) = \va(\K)$ then $H\in \va(\K)$.
\end{proof}

By Proposition \ref{mp1} we can define a prelinear Hilbert algebra
as a Hilbert algebra which satisfies the equation $\mathrm{(P)}$.

Let $H$ be a Hilbert algebra and $a,b\in H$. If there exists the
supremum of $\{a,b\}$ with respect to the natural order of $H$
then we write $a\vee b$ for this element. In what follows we give
some another characterizations for the variety of prelinear
Hilbert algebras (\cite[Theorem 5.1]{A. Monteiro 1996}).

\begin{prop} \label{mp2}
Let $H$ be a Hilbert algebra. Then $H$ is prelinear if and only if
for every $a,b\in H$ it holds that $(a\ra b)\vee (b\ra a) = 1$.
\end{prop}

The following definition will be used throughout the paper.

\begin{defn}\label{RS}
A poset $(X,\leq)$ is said to be a root system if $[x)$ is a chain
for each $x\in X$.
\end{defn}

The following proposition is \cite[Theorem 4.5]{A. Monteiro 1996}.

\begin{prop} \label{mp4}
Let $H$ be a Hilbert algebra. Then $H$ is prelinear if and only if
$X(H)$ is a root system.
\end{prop}

Let $H$ be a Hilbert algebra. It is part of the folklore of
Hilbert algebras that there exists an order isomorphism between
the lattice of congruences of $H$ and the lattice of implicative
filters of $H$. The isomorphism is established via the assignments
$\theta \ra 1/\theta$ and $F \ra \theta_F =\{(a,b)\in H\times
H:a\ra b\in F$ and $b\ra a \in F\}$. We write $H/F$ in place of
$H/\theta_F$, where $H/\theta_F$ is the set of equivalence classes
associated to the congruence $\theta_F$.

\begin{prop} \label{mp5}
Let $H$ be a Hilbert algebra.

The following conditions are equivalent:
\begin{enumerate}[\normalfont 1.]
\item $H$ is prelinear. \item For every $a,b \in H$ and $P\in
X(H)$, $a\ra b \in P$ or $b\ra a \in P$. \item For every $P\in
X(H)$, $A/P$ is a chain. \item For every $a,b,c \in H$ and $P\in
X(H)$, $l(a,b,c) \in P$.
\end{enumerate}
\end{prop}

\begin{proof}
In order to show that 1. implies 2., let $H$ be prelinear, $P\in
X(H)$ and $a,b \in H$. By Proposition \ref{mp2} we have that
$(a\ra b)\vee (b\ra a) = 1$. Then we obtain that $(a\ra b)\vee
(b\ra a) \in P$, so by \cite[Theorem 3.2]{CHJ} we deduce that
$a\ra b\in P$ or $b\ra a \in P$. Now we will see that 2. implies
1. Suppose that $H$ is not prelinear, so by Proposition \ref{mp4}
we have that there exist $P, Q\in X(H)$ such that $P\subseteq Q$,
$P\subseteq Z$, $Q\nsubseteq Z$ and $Z\nsubseteq Q$. Thus, there
are $a,b \in H$ such that $a\in Q$, $a\notin Z$, $b\in Z$ and
$b\notin Q$. This implies that $a\ra b\notin P$ and $b\ra a\notin
P$, which is a contradiction.

The equivalence between 2. and 3. is immediate.

The fact that 1. implies 4. follows from Proposition \ref{mp1}.
Now we will prove that 4. implies 1. Suppose that $H$ is not
prelinear. Hence, by Proposition \ref{mp4} there exist $Q,Z\in
X(H)$ such that $P\subseteq Q$, $P\subseteq Z$, $Q\nsubseteq Z$
and $Z\nsubseteq Q$. Hence, there exist $a,b\in H$ such that $a\in
Q$, $a\notin Z$, $b\in Z$ and $b\notin Q$. Thus, $a,b\notin P$.
Notice that $a\ra b\notin P$, otherwise $a\ra b\in Q$ and since
$a\in Q$ then $b\in Q$, which is impossible. Similarly we deduce
that $b\ra a\notin P$. Since $P$ is irreducible then by Lemma
\ref{irred} we have that there exists $c\in A$ such that $(a\ra
b)\ra c\in P$, $(b\ra a)\ra c\in P$ and $c\notin P$. Taking into
account that $l(a,b,c)\in P$, applying modus ponens twice we get
that $c\in P$, which is a contradiction. Thus, we have proved 4.
\end{proof}




We assume that the reader is familiar with the theory of Heyting
algebras \cite{BD}. Prelinear Heyting algebras were considered by
Horn in \cite{Horn} as an intermediate step between the classical
calculus and intuitionistic one and they were studied also by
Monteiro \cite{A. Monteiro 1980}, G. Martínez \cite{Guille} and
others. This is the subvariety of Heyting algebras generated by
the class of totally ordered Heyting algebras and can be
axiomatized by the usual equations for Heyting algebras plus the
prelinearity law $(x\rightarrow y)\vee (y\rightarrow x) = 1$. In
(\cite{BD}, ch. IX) and in \cite{A. Monteiro 1980} there are
characterizations for prelinear Heyting algebras. Horn showed in
\cite{Horn} (although it was in fact proved before by Monteiro,
see \cite{A. Monteiro 1980}) that prelinear Heyting algebras can
be characterized among Heyting algebras in terms of the prime
filters. More precisely, a Heyting algebra $H$ is prelinear if and
only if the poset of prime filters of $H$ with the inclusion as
order is a root system.

Prelinear Heyting algebras, under the name of G\"odel algebras, are a 
particular class of t-norm based algebras of great interest for fuzzy 
logic \cite{Hajek}.

We also assume that the reader is familiar with the theory of
implicative semilattices \cite{N}. It is known that every
implicative element has a largest element with respect to the
order associated to its unserlying semilattice. We denoted by $1$
to this element. A \emph{bounded implicative smilattice} is an
algebra $(H,\we,\ra,0,1)$ of type $(2,2,0,0)$ such that
$(H,\we,\ra)$ is an implicative semilattice and $0$ is the minimum
element of $(H,\we)$. In what follows we establish some
connections between implicative semilattices, prelinear Heyting
algebras and prelinear Hilbert algebras.

The following result is known and can be deduced from the papers
of W.C. Nemitz \cite{N}, A. Monteiro \cite{A. Monteiro 1996} and
T. Katri\v{n}\'{a}k \cite{K}. Here we give a proof for
completeness.

\begin{prop} \label{ip}
Let $(H,\ra,\we,1)$ be an implicative semilattice such that
$(H,\ra,1)$ is a prelinear Hilbert algebra. Then for every $a,b\in
H$ there exist $a\vee b$ and it is given by
\[
a\vee b=((a\ra b)\ra b)\we((b\ra a)\ra a).
\]
Moreover, if $(H,\ra,\we,0,1)$ is a bounded implicative
semilattice then the algebra $(H,\ra,\we,\vee,0,1)$ is a prelinear
Heyting algebra.
\end{prop}

\begin{proof}
Let $a,b\in H$. By Proposition \ref{mp4} there exists the supremum
of the set $\{a\ra b,b\ra a\}$. Moreover, $(a\ra b)\vee(b\ra
a)=1$. Consider the element $c=((a\ra b)\ra b)\wedge((b\ra a)\ra
a)$. It is clear that $a\leq c$ and $b\leq c$. So, $c$ is an upper
bound of $\{a, b\}$. Let $z\in H$ such that $a\leq z$ and $b\leq
z$. We will prove that $c\leq(a\ra b)\ra z$ and $c\leq(b\ra a)\ra
z$. Suppose that $c\nleq(a\ra b)\ra z$ or $c\nleq(b\ra a)\ra z$.
In the first case, it follows from Corollary \ref{tfpc2} that
there exists $P\in X(H)$ such that $c\in P$ and $(a\ra b)\ra
z\notin P$. So, by Corollary \ref{tfpc3} there exists $Q\in X(H)$
such that $a\ra b\in Q$, $z\notin Q$ and $P\subseteq Q$. Since
$c\in P$ then $(a\ra b)\ra b\in P\subseteq Q$. By modus ponens we
have that $b\in Q$, and since $b\leq z$ then we get $z\in Q$,
which is a contradiction. The other case is similar. Thus,
$c\leq(a\ra b)\ra z$ and $c\leq(b\ra a)\ra z$, i.e., $a\ra b\leq
c\ra z$ and $b\ra a\leq c\ra z$. Then, $(a\ra b)\vee(b\ra a)=1\leq
c\ra z$. Therefore, $c\leq z$. Therefore, there exist $a\vee b$
and it is given by $a\vee b=((a\ra b)\ra b)\we((b\ra a)\ra a)$.

The rest of the proof follows from Proposition \ref{mp2}.
\end{proof}

\section{An adjunction}\label{s3}

Let $\Hil$ be the algebraic category of Hilbert algebras, $\PHil$
the algebraic category of prelinear Hilbert algebras and $\PHilz$
the algebraic category of bounded prelinear Hilbert algebras. In
this section we give an explicit description of the left adjoint
of the forgetful functor from $\PHey$ to $\PHilz$.

We start with some preliminary definitions and properties
involving the duality developed in \cite{CCM} for the algebraic
category of Hilbert algebras (see also \cite{CMs}). \vspace{3pt}

If $f:H\ra G$ is a function between Hilbert algebras, we define
the relation $R_f \subseteq X(G)\times X(H)$ by
\[
(P,Q)\in R_f\;\text{if and only if}\; f^{-1}(P) \subseteq Q.
\]

If $H$ is a Hilbert algebra and $a\in H$ we define
\begin{equation} \label{var}
\varphi_{H}(a): = \{P\in X(H): a\in P\}.
\end{equation}

If there is not ambiguity we write $\varphi$ in place of
$\varphi_H$. If $X$ is a set then we define $Y^{c}:= \{x\in X:
x\notin Y\}$.

\begin{rem} \label{HH}
Let $(X,\leq)$ be a poset. Write $X^{+}$ for the set of upsets of
$(X,\leq)$. Define on $X^{+}$ the binary operation $\Ra$ by
\begin{equation}\label{imp}
U\Rightarrow V: =(U\cap V^{c}]^{c}.
\end{equation}
Then $(X^{+},\cap,\cup,\Ra,\emptyset,X)$ is a complete Heyting
algebra. If there is not ambiguity, we also write $X^{+}$ for this
Heyting algebra. In particular, $X^{+}$ can be seen as a Hilbert
algebra.
\end{rem}

Let $(X,\tau)$ be a topological space. An arbitrary non-empty
subset $Y$ of $X$ is said to be \emph{irreducible} if for any
closed subsets $Z$ and $W$ such that $Y\subseteq Z \cup W$ we have
that $Y \subseteq Z$ or $Y \subseteq W$. We say that $(X,\tau)$ is
\emph{sober} if for every irreducible closed set $Y$ there exists
a unique $x\in X$ such that $Y = \clx$, where $\clx$ denotes the
clausure of $\{x\}$. A subset of $X$ is \emph{saturated} if it is
an intersection of open sets. The \emph{saturation} of a subset
$Y$ of $X$ is defined as $\sat(Y):= \bigcap\{U\in \tau: Y\subseteq
U\}$. Also recall that the \emph{specialization order} of
$(X,\tau)$ is defined by $x\leq y$ if and only if $x \in \cly$.
The relation $\leq$ is reflexive and transitive, i.e., a
quasi-order. The relation $\leq$ is a partial order if $(X,\tau)$
is $T_0$. The dual quasi-order of $\leq$ will be denoted by $\pd$.
Hence,
\[
x\pd y\;\text{if and only if}\; y\in \clx.
\]

Let $(X,\tau)$ be a topological space which is $T_0$, and consider
the order $\pd$. Let $x\in X$ and $Y\subseteq X$. Then $\clx =
[x)$ and $\sat(Y) = (Y]$.

\begin{defn} \label{Hs}
A Hilbert space, or $H$-space for short, is a structure
$(X,\tau,\KA)$ where $(X,\tau)$ is a topological space, $\KA$ is a
family of subsets of $X$ and the following conditions are
satisfied:
\begin{enumerate} [\normalfont (H1)]
\item $\KA$ is a base of open and compact subsets for the topology
$\tau$ on $X$. \item For every $U, V \in \KA$, $\sat(U \cap V^{c})
\in \KA$. \item $(X, \tau)$ is sober.
\end{enumerate}
\end{defn}

In what follows, if $(X,\tau,\KA)$ is an $H$-space we simply write
$(X,\KA)$.

\begin{rem} \label{rimp}
\begin{enumerate} [\normalfont 1.]
\item A sober topological space is $T_0$. \item Viewing any
topological space as a poset, with the order $\pd$, condition
$\mathrm{(H2)}$ of Definition \ref{Hs} can be rewritten as: for
every $U,V\in \KA$, $(U\cap V^{c}] \in \KA$.
\end{enumerate}
\end{rem}

If $X, Y$ are sets and $R\subseteq X\times Y$, take $R(x): =\{y\in
Y:(x,y)\in R\}$, and if $U\subseteq Y$, $R^{-1}(U): =\{x\in X:
R(x)\cap U \neq \emptyset\}$.

\begin{defn} \label{defHR}
Let $\X_1 = (X_1,\KA_1)$ and $\X_2 = (X_2,\KA_2)$ be two
$H$-spaces. Let us consider a relation $R\subseteq X_1 \times
X_2$. We say that $R$ is an $H$-relation from $\X_1$ into $\X_2$
if it satisfies the following properties:
\begin{enumerate} [\normalfont (HR1)]
\item $R^{-1}(U) \in \KA_1$, for every $U \in \KA_2$. \item $R(x)$
is a closed subset of $\X_2$, for all $x\in X_1$.
\end{enumerate}
We say that $R$ is an $H$-functional relation if it satisfies the
following additional condition:
\begin{enumerate}
\item[$\mathrm{(HF)}$] If $(x, y)\in R$ then there is $z\in X_1$
such that $z\in \clx$ and $R(z) = \cly$.
\end{enumerate}
\end{defn}

\begin{rem}
Condition $\mathrm{(HF)}$ from Definition \ref{defHR} can also be
given as follows: if $(x, y)\in R$ then there exists $z\in X_1$
such that $x\pd z$ and $R(z) = [y)$.
\end{rem}

If $H$ is a Hilbert algebra then $\X(H) = (X(H), \KA_H)$ is an
$H$-space, where $\KA_H: =\{\varphi(a)^{c}: a\in H\}$. If $f$ is a
morphism in $\Hil$ then $R_f$ is an $H$-functional relation. Write
$\HS$ for the category whose objects are Hilbert spaces and whose
morphisms are $H$-functional relations. The assignment $H\mapsto
\X(H)$ can be extended to a functor $\X:\Hil \ra \HS$.

\begin{rem} \label{dso}
If $H\in \Hil$ and $P,Q\in \X(H)$, then $P\subseteq Q$ if and only
if $P\leq_{d} Q$.
\end{rem}

Let $(X,\KA)$ be an $H$-space. Define $D(X): =\{U\subseteq
X:U^{c}\in \KA\}$. Then $D(X) \subseteq X^{+}$. It follows from
Definition \ref{Hs} and Remark \ref{rimp} that $D(X)$ is closed
under the operation $\Ra$ given in (\ref{imp}) of Remark \ref{HH}.
Since $X^{+}$ is a Heyting algebra then $\D(X) = (D(X),
\Rightarrow, X)$ is a Hilbert algebra. If $R$ is an $H$-functional
relation from $(X_1, \KA_1)$ into $(X_2, \KA_2)$, then the map
$h_R: \D(X_2) \ra \D(X_1)$ given by $h_R(U) = \{x\in
X_1:R(x)\subseteq U\}$ is a morphism in $\Hil$. The assignment
$X\mapsto \D(X)$ can be extended to a functor $\D:\HS\ra \Hil$.
Finally, if $H\in \Hil$ then the map $\varphi: H \ra \D(\X(H))$
defined as in $(\ref{var})$ is an isomorphism in $\Hil$.

Recall that if $(X,\tau,\KA)$ is an $H$-space then $(X,\pd)$ is a
poset, where $\pd$ is the dual of the specialization order
associated to the topological space $(X,\tau)$. Also recall that
if $H$ is a Hilbert algebra then the dual of the specialization
order associated to $\X(H)$ is the inclusion (Remark \ref{dso}).
We say that there is an order isomorphism between two $H$-spaces
if and only if there is an order isomorphism between its
associated posets obtained through the dual of the specialization
order.

If $(X,\tau), (Y,\sigma)$ be topological spaces and $f: X\ra Y$a
function. Then we define the binary relation $f^{R}$ by
\[
(x,y)\in f^{R}\;\text{if and only if}\; f(x)\pd y.
\]
If $(X,\KA)$ is an $H$-space, then the map $\epsilon_{X}: X\ra
\X(\D(X))$ given by $\epsilon_{X}(x) =\{U\in D(X): x\in U\}$ is an
order isomorphism and a homeomorphism between the topological
spaces $X$ and $\X(\D(X))$(\cite[Theorem 2.2]{CCM}). If there is
not ambiguity we will write $\epsilon$ in place of $\epsilon_{X}$.
Moreover, the relation $\epsilon^{R}\subseteq X\times X(D(X))$ is
given by $(x,P)\in \epsilon^{R}$ if and only if
$\epsilon(x)\subseteq P$. Moreover, $\epsilon^{R}$ is an
$H$-functional relation which is an isomorphism in $\HS$.

\begin{thm}  \label{rh}
The contravariant functors $\X$ and $\D$ define a dual equivalence
between $\Hil$ and $\HS$ with natural equivalences $\epsilon^{R}$
and $\varphi$.
\end{thm}

If $H$ is a bounded Hilbert algebra then $\varphi(0) = \emptyset$,
so $\X(H)\in \KA_{H}$. Conversely, if $(X,\KA)$ is an $H$-space
such that $X\in \KA$ then $\D(X)$ is a bounded Hilbert algebra. If
$H$, $G$ are bounded Hilbert algebras and $f:H\ra G$ is a morphism
of Hilbert algebras such that $f(0) = 0$ then $R_{f}(P)\neq
\emptyset$ for every $P\in \X(G)$. Conversely, if $(X_1, \KA_1)$
and $(X_2, \KA_2)$ are $H$-spaces such that $X_1\in \KA_1$,
$X_2\in \KA_2$ and $R$ is an $H$-functional relation from $(X_1,
\KA_1)$ into $(X_2, \KA_2)$ such that $R(x) \neq \emptyset$ for
every $x\in X_1$, then $h_R: \D(X_2) \ra \D(X_1)$ satisfies that
$h_{R}(\emptyset) = \emptyset$. Moreover, if $(X, \KA)$ is an
$H$-space such that $X\in \KA$ then $\epsilon(x) \neq \emptyset$
for every $x\in X$.

Let $\PHSz$ be the full subcategory of $\HS$ whose objects
$(X,\KA)$ are such that $(X,\leq_{d})$ is a root system and $X\in
\KA$, and whose morphisms are $H$-functional relations $R$ such
that $R(x) \neq \emptyset$ for every $x$. Straightforward
computations based in Theorem \ref{rh} and Proposition \ref{mp4}
proves the following result.

\begin{cor} \label{corrh}
There exists a categorical equivalence between $\PHilz$ and
$\PHSz$.
\end{cor}

\subsection*{The adjunction}

Let $\PHey$ be the algebraic category of prelinear Heyting
algebras.

\begin{lem} \label{bpteo}
Let $H \in \PHil$. Then $\X(H)^{+}\in \PHey$.
\end{lem}

\begin{proof}
Let $U,V \in \X(H)^{+}$. We need to show that $(U\Ra V) \cup (V\Ra
U) = \X(H)$, or, equivalently, that $(U\cap V^c] \cap (V\cap U^c]
= \emptyset$. Suppose that there is $P \in (U\cap V^c] \cap (V\cap
U^c]$. Thus, there exist $Q, Z\in \X(H)$ such that $P\subseteq Q$,
$P\subseteq Z$, $Q\in U\cap V^c$ and $Z\in V\cap U^c$. By
Proposition \ref{mp4} we have that $\X(H)$ is a root system, so
$Q\subseteq Z$ or $Z\subseteq Q$. Assume that $Q\subseteq Z$.
Since $Q\in U$ then $Z\in U$, which is a contradiction because
$Q\in U^c$. Analogously, we obtain a contradiction assuming that
$Z\subseteq Q$. Therefore, $\X(H)^{+}$ is a prelinear Heyting
algebra.
\end{proof}

If $S$ is a subset of a Heyting algebra $H$, write $\langle S
\rangle_{\Hey}$ for the Heyting subalgebra of $H$ generated by
$S$. Let $H\in \Hil$. Since $\varphi(H)$ is a subset of the
Heyting algebra $\X(H)^+$ then we define the following Heyting
subalgebra of $X(H)^+$:
\[
H^*: = \langle \varphi(H) \rangle_{\Hey}.
\]

\begin{cor}\label{bpteo2}
If $H\in \PHil$ then $H^*\in \PHey$.
\end{cor}

\begin{proof}
It follows from Lemma \ref{bpteo} and the fact that the class of
prelinear Heyting algebras is a variety.
\end{proof}

In what follows we study the link between morphisms in $\PHilz$
and morphisms in $\PHey$.

\begin{lem}\label{aux}
Let $f:H\ra G$ be a morphism in $\PHilz$ and $P\in \X(G)$. Then
$f^{-1}(P)\in \X(H)$.
\end{lem}

\begin{proof}
Let $a, b\in G$. By Proposition \ref{mp5} we have that $f(a)\ra
f(b)\in P$ or $f(b) \ra f(a) \in P$. Since $f(a\ra b) = f(a) \ra
f(b)$ and $f(b\ra a) = f(b) \ra f(a)$ then $a\ra b \in f^{-1}(P)$
or $b\ra a \in f^{-1}(P)$. This property will be used in order to
show that $f^{-1}(P)\in \X(H)$.

First note that since $0\notin P$ then $f^{-1}(P)$ is a proper
implicative filter. Now we will prove that $f^{-1}(P)$ is
irreducible. Let $f^{-1}(P) = F_1\cap F_2$ with $F_1,F_2\in
\X(H)$. Suppose that $F_1\nsubseteq f^{-1}(P)$ and $F_2 \nsubseteq
f^{-1}(P)$. Then there exist $a,b\in H$ such that $a\in F_1,
f(a)\notin P$, $b\in F_2$ and $f(b)\notin P$. Since $a\ra b\in
F_1$ and $a\in F_1$ then $b\in F_1$, so $b\in F_1 \cap F_2 =
f^{-1}(P)$, i.e., $f(b) \in P$, which is a contradiction. Thus,
$f^{-1}(P) = F_1$ or $f^{-1}(P) = F_2$. Therefore, $f^{-1}(P) \in
\X(H)$.
\end{proof}

Let $f:H\ra G$ be a morphism in $\PHilz$. Taking into account the
commutativity of the the following diagram
$$
 \xymatrix{
   H \ar[rr]^{\varphi} \ar[d]_{f} & & \D(\X(H))  \ar[d]^{\D(\X(f))}\\
   G \ar[rr]^{\varphi} & &  \D(\X(G)) \
   }
$$
we obtain that $\varphi(f(a)) = \D(\X(f))(\varphi(a))$. Besides,
in \cite[Lemma 3.3]{CCM} it was proved that if $P\in \X(H)$ then
$f(a)\in P$ if and only if for all $Q\in \X(G)$, if $(P,Q) \in
R_f$ then $a\in Q$, i.e., $f(a) \in P$ if and only if $R_f(P)
\subseteq \varphi(a)$. Thus, for every $a\in H$ we have that
\begin{equation} \label{cd}
\D(\X(f))(\varphi(a)) = \{P\in \X(G): R_{f}(P)\subseteq
\varphi(a)\}.
\end{equation}

\begin{cor} \label{ext}
Let $f:H\ra G$ be a morphism in $\PHilz$ and $g = \D(\X(f))$. The
morphism $g$ in $\PHilz$ can be extended to the homomorphism of
Heyting algebras $\hat{g}: \X(H)^{+} \ra \X(G)^{+}$ given by
\[
\hat{g}(U) = \{P\in \X(G): R_{f}(P)\subseteq U\}.
\]
\end{cor}

\begin{proof}
The proof is similar to \cite[Lemma 8]{CSM} by considering Lemma
\ref{aux}.
\end{proof}

The following remark is a well known fact from universal algebra
\cite{BS}.

\begin{rem} \label{r1}
Let $A$ and $B$ be algebras of the same type and $X\subseteq A$.
Write $\Sg(X)$ for the subalgebra of $A$ generated by $X$ and
$\Sg(f(X))$ for the subalgebra of $B$ generated by $f(X)$. If
$f:A\ra B$ is a homomorphism then $f(\Sg(X))= \Sg(f(X))$.
\end{rem}

\begin{lem} \label{morf}
The homomorphism of Heyting algebras $\hat{g}$ defined in
Corollary \ref{ext} satisfies $\hat{g}(H^*) \subseteq \G^*$.
\end{lem}

\begin{proof}
It follows from Lemma \ref{ext}, Remark \ref{r1} and the equality
$g(\varphi(a)) = \varphi(f(a))$ given in (\ref{cd}).
\end{proof}

Let $f:H\ra G$ be a morphism in $\PHilz$.  It follows from
Corollary \ref{ext} and Lemma \ref{morf} that the map $f^{*}:
H^{*} \ra G^{*}$ given by $f^{*}(U) = \hat{g}(U)$ is a morphism in
$\PHey$.

Let $\mathsf{Id}$ be an identity morphism in $\PHilz$. It is
immediate that $\mathsf{Id}^{*}$ is an identity in $\PHey$. Let
$f:H\ra G$ and $g:G \ra K$ be morphisms in $\PHilz$. It follows
from \cite[Theorem 3.3]{CCM} that $R_{g\circ f} = R_{g} \circ
R_{f}$. Hence, straightforward computations based in the above
mentioned equality shows that
\begin{equation*}
\label{isafunctor} (g \circ f)^{*} = g^* \circ f^*.
\end{equation*}

Hence we have that the assignment $H\mapsto H^*$ and $f\mapsto
f^*$ defines a functor $\f:\PHilz \ra \PHey$. \vspace{3pt}

Let $\U$ be the forgetful functor from $\PHey$ to $\PHilz$. Let
$H\in \PHilz$. Consider the injective morphism of Hilbert algebras
$\psi:H\ra \U(H^*)$ given by $\psi(a) = \varphi(a)$.

\begin{prop} \label{pu}
Let $G\in \PHey$ and $f:H\ra U(G^*)\in \PHilz$. Then, there exists
a unique morphism $h:H^* \ra G$ such that $f = \U(h) \circ \psi$.
\end{prop}

\begin{proof}
The map $f^*: H^* \ra G^*$ is a morphism in $\PHey$. Since $G\in
\PHey$ then for every $a,b\in G$ we have that $\varphi(a\we b) =
\varphi(a) \cap \varphi(b)$, so we deduce that the map
$\varphi:G\ra G^*$ is an isomorphism in $\PHey$. Hence, the map
$h:H^* \ra G$ given by $h = \varphi^{-1} \circ f^*$ is also a
morphism in $\PHey$. Finally, taking into account (\ref{cd}) we
have that $f = \U(h) \circ \psi$.
\end{proof}

Let $\mathrm{I}$ be the identity functor in $\PHilz$. It follows
from (\ref{cd}) that $\Psi: \mathrm{I} \ra \U \circ \f$ is a
natural transformation. Here, the family of morphism associated to
the natural transformation is given by the morphisms $\psi$.

In other words, to say that $\Psi: \mathrm{I} \ra \U \circ \f$ is
a natural transformation is equivalent to say that if $f:H\ra G$
is a morphism in $\PHilz$ then the following diagram commutes:

$$
 \xymatrix{
   H \ar[rr]^{f} \ar[d]_{\psi} & & G  \ar[d]^{\psi}\\
   \U(H^*) \ar[rr]^{\U(f^*)} & & \U(G^*). \
   }
$$

Therefore we get the following result.

\begin{thm} \label{pteo}
The functor $\f: \PHilz \ra \PHey$ is left adjoint to $\U$.
\end{thm}

In \cite{CJ2} Celani and Jansana presented an explicit description
for the left adjoint of the forgetful functor from the category of
implicative semilattices to the category of Hilbert algebras. In
particular, they proved that if $H$ is a Hilbert algebra then
\[
\FC(H) = \{U : U = \varphi(a_1)\cap \cdots \cap
\varphi(a_n)\;\text{for some}\; a_1,\ldots,a_n \in H\}.
\]
is an implicative semilattice, which is called the \emph{free
implicative semilattice extension of the Hilbert algebra $H$} (see
also \cite{CSM}).

\begin{prop} \label{SH}
Let $H\in \PHilz$. Then $S(H) \in \PHey$ and $S(H) = H^*$.
\end{prop}

\begin{proof}
By Proposition \ref{mp4} and Proposition \ref{ip} in order to
prove that $S(H) \in \PHey$ it is enough to see that
$(U\Rightarrow V)\cup(V\Rightarrow U)=X(H)$ for every $U,V\in
S(H)$. Let $U,V\in S(H)$. Then there exist two finite subsets
$\{a_{1},\ldots,a_{n}\}$ and $\{b_{1},\ldots,b_{k}\}$ of $H$ such
that $U=\varphi(a_{1})\cap\ldots\cap\varphi(a_{n})$ and
$V=\varphi(b_{1})\cap\ldots\cap\varphi(b_{k})$. Suppose that there
exists $P\in \X(H)$ such that $P\notin(U\Rightarrow
V)\cup(V\Rightarrow U)$. Thus, there exist $Q, Z\in \X(H)$ such
that $P\subseteq Q$, $P\subseteq Z$, $Q\in U\cap V^c$ and $Z \in
V\cap U^c$. Hence, there exists $i=1,\ldots,n$ and $j=1,\ldots,k$
such that $Q\notin\varphi(b_{j})$ and $Z\notin\varphi(a_{i})$,
i.e., $b_j\notin Q$ and $a_i \notin Z$. Since $H$ is prelinear
then it follows from Proposition \ref{mp4} that $Q\subseteq Z$ or
$Z\subseteq Q$. Then, $a_{i}\in Z$ or $b_{j}\in Q$, which is a
contradiction. Thus, $S(H) \in \PHey$.

The fact that $S(H) = H^*$ follows from that $S(H)$ is a Heyting
algebra and $\varphi(H) \subseteq S(H) \subseteq H^*$.
\end{proof}

Finally we will study some connections between $\X(H)$ and
$\X(H^*)$ for $H$ a finite Hilbert algebra.
\vspace{3pt}

\begin{prop}
Let $H$ be a finite Hilbert algebra. Then $H^* = \X(H)^{+}$.
\end{prop}

\begin{proof}
In order to prove it, let $U\in \X(H)^{+}$ and $U \neq \emptyset$.
Then there exist $P_1,\ldots,P_n \in \X(H)$ such that $U =
\bigcup_{i=1}^{n} [P_i)$. For instance, we can choose
$P_1,\ldots,P_n$ as the minimal elements of $U$. For every $i=1,
\ldots, n$ there exist $a_{i 1},\ldots,a_{i m} \in H$ such that
$P_i = \{a_{i 1},\ldots,a_{i m}\}$. So, $[P_i) =
\bigcap_{j=1}^{m}\varphi(a_{i j})$. Notice that for every
$P_1,\ldots, P_n$ we can choose the same $m$ in the above
mentioned reasoning. In order to make it possible, for every
$i=1,\ldots,n$ let $c_i$ be the cardinal of $P_i$. Write $m$ for
the maximum of the set $\{c_1, \ldots, c_n\}$. Then for each $c_i$
such that $c_i < m$ define $a_{ij} = 1$ for every $j \in
\{c_{i}+1,\ldots,m\}$, which was our aim. Hence, for every $U\in
\X(H)^{+}$ ($U \neq \emptyset$) there exists a finite family
$\{a_{ij}\}_{i,j}$ in $H$ (with $i=1,\ldots, n$ and $j=1,\ldots
m$) such that
\begin{equation} \label{finite}
U = \bigcup_{i=1}^{n}\bigcap_{j=1}^{m} \varphi(a_{ij}).
\end{equation}
Therefore, $H^{*} = \X(H)^{+}$.
\end{proof}

The following result will be used later.

\begin{prop}\label{finiteH}
Let $H$ be a finite algebra of $\PHilz$. Then the map $\eta:
\X(H^*) \ra \X(H)$ given by $\eta(P) = \varphi^{-1}(P)$ is an
order isomorphism.
\end{prop}

\begin{proof}
It follows from \cite[Theorem 11]{CSM}, \cite[Lemma 4.6]{CJ2} and
\cite[Proposition 7.7]{CJ2}.
\end{proof}

\section{Coproduct in $\PHilz$ for finite algebras}\label{s4}

In this section we study some properties of the coproduct of two
finite algebras in $\PHilz$ through the study of the product in
$\PHSz$ of its associated finite $H$-spaces.

Let $H\in \PHey$ and $S\subseteq H$ with $0\in S$. We write
$\langle S\rangle_{\PHilz}$ as the Hilbert subalgebra of $H$
generated by $S$. Note that $\langle S\rangle_{\PHilz}\in \PHilz$.

Let $H,G\in \PHilz$. Let $i_H:H^* \ra H^* \cPHey G^*$ and $i_G:G^*
\ra H^* \cPHey G^*$ be the morphisms in $\PHey$ given by the
definition of coproduct. Consider $\varphi_H:H\ra H^*$ and
$\varphi_G:G\ra G^*$. Then we define
\[
\eta_{GH}: = i_{H}(\varphi_{H}(H)) \cup i_{G}(\varphi_{G}(G)).
\]
Consider maps $j_H:H \ra \langle \eta_{GH}\rangle_{\PHilz}$ and
$j_G:G \ra \langle \eta_{GH}\rangle_{\PHilz}$ defined by by $j_H =
i_{H} \circ \varphi_H$ and $j_G = i_{G} \circ \varphi_G$.

Let $H, G\in \PHilz$. In the following proposition we will show
that for every $\alpha:H\ra J \in \PHilz$ and $\beta:G\ra J \in
\PHilz$ there is an unique morphism $f:\langle
\eta_{GH}\rangle_{\PHilz} \ra J$ in $\PHilz$ such that the
following diagram commutes:

$$
\xymatrix{
& J  & \\
H \ar[ur]_{\alpha} \ar[dr]_{j_H} & & G \ar[ul]^{\beta} \ar[dl]^{j_G}\\
& \langle \eta_{GH}\rangle \ar[uu]^{f} & }
$$

\begin{prop}
Let $H, G\in \PHilz$. Then $H\cPHilz G \cong \langle
\eta_{GH}\rangle_{\PHilz}$.
\end{prop}

\begin{proof}
Let $\alpha:H\ra J \in \PHilz$ and $\beta:G\ra J \in \PHilz$. Then
$\alpha^*:H^*\ra J^* \in \PHey$ and $\beta^*:G^*\ra J^* \in
\PHey$. Then there exists an unique morphism $F:H^* \cPHey G^* \ra
J^*$ in $\PHey$ such that the following diagram commutes:
$$
\xymatrix{
& J^*  & \\
H^* \ar[ur]_{\alpha^*} \ar[dr]_{i_H} & & G^* \ar[ul]^{\beta^*} \ar[dl]^{i_G}\\
& H^*\cPHey G^* \ar[uu]^{F} & }
$$
In particular,
\begin{equation} \label{cop1}
\alpha^* = F\circ i_H,
\end{equation}
\begin{equation} \label{cop2}
\beta^* = F\circ i_G.
\end{equation}
In what follows we will see that $F(\langle
\eta_{GH}\rangle_{\PHilz}) \subseteq \varphi_J(J)$. First note
that since $F$ is a morphism of Heyting algebras then $F$ is also
a morphism of Hilbert algebras, so $F(\langle
\eta_{GH}\rangle_{\PHilz}) = \langle
F(\eta_{GH})\rangle_{\PHilz}$. Taking into account that $\varphi_J
\circ \alpha = \alpha^* \circ \varphi_H$, $\varphi_J \circ \beta =
\beta^* \circ \varphi_G$, (\ref{cop1}) and (\ref{cop2}) we obtain
that
\[
\begin{array}
[c]{lllll}
 F(\eta_{GH}) & = & F(i_H(\varphi_H(H))) \cup F(i_G(\varphi_G(G)))  &  & \\
 & = & (F\circ i_H)(\varphi_H(H)) \cup (F\circ i_G)(\varphi_G(G))&  &\\
 & = & (\alpha^* \circ \varphi_H)(H) \cup (\beta^* \circ \varphi_G)(G)&  &\\
 & = & \varphi_J(\alpha(H)) \cup \varphi_J(\beta(G))&  &\\
 & = & \varphi_J(\alpha(H)\cup \beta(G))&  &\\
 & \subseteq & \varphi_J(J). &  &
\end{array}
\]
Thus, $F(\langle \eta_{GH}\rangle_{\PHilz}) = \langle
F(\eta_{GH})\rangle_{\PHilz} \subseteq \langle
\varphi_J(J)\rangle_{\PHilz} = \varphi_J(J)$, i.e., $F(\langle
\eta_{GH}\rangle) \subseteq \varphi_J(J)$, which was our aim.

We also write $\varphi_J:J\ra \varphi_J(J)$. Since $\varphi_J$ is
an isomorphism in $\PHilz$, we define $f:\langle
\eta_{GH}\rangle_{\PHilz} \ra J$ by $f = \varphi_{J}^{-1} \circ
F$. We will prove that $f$ is the unique morphism in $\PHilz$ such
that the following diagram commutes:
$$
\xymatrix{
& J  & \\
H \ar[ur]_{\alpha} \ar[dr]_{j_H} & & G \ar[ul]^{\beta} \ar[dl]^{j_G}\\
& \langle \eta_{GH}\rangle_{\PHilz} \ar[uu]^{f} & }
$$
In order to show it, note that

\vspace{3pt}
\[
\begin{array}
[c]{lllll}
f \circ j_H & = & \varphi_{J}^{-1} \circ F \circ i_H \circ \varphi_H  &  & \\
 & = & \varphi_{J}^{-1} \circ \alpha^* \circ \varphi_H &  &\\
 & = & \varphi_{J}^{-1} \circ \varphi_J \circ \alpha&  &\\
 & = & \alpha. &  &
\end{array}
\]
\vspace{3pt}

Thus, $f \circ j_H = \alpha$. Similarly it can be proved that
$f\circ j_G = \beta$. Finally, let $g:\langle
\eta_{GH}\rangle_{\PHilz} \ra J$ be a morphism in $\PHilz$ such
that $g \circ j_H = \alpha$ and $g\circ j_G = \beta$. Since $f = g
$ in $\eta_{GH}$ then $f = g$ in $\langle
\eta_{GH}\rangle_{\PHilz}$. Therefore, $H\cPHilz G \cong \langle
\eta_{GH}\rangle_{\PHilz}$.
\end{proof}

Let $\Pos$ be the category of posets and $\Esa$ the category of
Esakia spaces. We assume that the reader is familiar with the
theory of Esakia spaces and the fact that there exists a
categorial equivalence between the algebraic category of Heyting
algebras and the category whose objects are Esakia spaces and
whose morphisms $f:(X,\leq,\tau) \ra (Y,\sigma,\leq)$ are
continuos order-preserving maps which satisfy that $f^{-1}([x)) =
[f^{-1}(\{x\}))$ for every $x\in X$ \cite{Es,Mor} (see also
\cite{Pr1,Pr2}). In particular, there exists a categorical
equivalence btween $\PHey$ and $\Esap$, where $\Esap$ is the full
subcategory of $\Esa$ whose objects are root systems.

It is part of the folklore of prelinear Heyting algebras that the
coproduct of finite algebras in $\PHey$ is finite. This fact will
be used in order to show that the coproduct of finite algebras in
$\PHilz$ is also finite.

\begin{lem} \label{lemfin}
Let $H,G$ be finite algebras in $\PHilz$. Then the algebra $H
\cPHilz G$ is finite.
\end{lem}

\begin{proof}
Consider the monomorphism $i: H\cPHilz G \ra \U((H\cPHilz G)^*)$
in $\PHilz$ given by $i(x) = x$. Taking into account Theorem
\ref{pteo} we have that $(H\cPHilz G)^* \cong H^* \cPHey G^*$ in
$\PHey$, so there exists a monomorphism
\[
j:H\cPHilz G \ra \U(H^* \cPHey G^*)
\]
in $\PHilz$. Since $H^*$ and $G^*$ are finite then $H^* \cPHey
G^*$ is finite. Therefore, $H\cPHilz G$ is finite.
\end{proof}

\begin{lem} \label{rp1}
Let $H,G$ be finite algebras in $\PHilz$. Then there exists an
order isomorphism between $\X(H) \prod_{\PHSz} \X(G)$ and $\X(H^*)
\prod_{\Esap} \X(G^*)$.
\end{lem}

\begin{proof}
By Corollary \ref{corrh} we have that
\[
\X(H) \prod_{\PHSz} \X(G) \cong \X(H\cPHilz G)
\]
in $\HS$, so it follows from \cite[Theorem 3.2]{CCM} that
\begin{equation}\label{c0}
\X(H) \prod_{\PHSz} \X(G) \cong \X(H\cPHilz G)\;\text{in}\;\Pos.
\end{equation}
Since $H$ and $G$ are finite then by Lemma \ref{lemfin} we have
that $H\cPHilz G$ is finite. Thus, it follows from Proposition
\ref{finiteH} that
\begin{equation} \label{c1}
\X(H\cPHilz G) \cong \X((H\cPHilz G)^{*})\;\text{in}\; \Pos.
\end{equation}

By Theorem \ref{pteo} we have that $(H\cPHilz G)^* \cong H^*
\cPHey G^*$ in $\PHey$. Thus, we obtain that $\X((H\cPHilz G)^*)
\cong \X(H^*) \prod_{\Esap} \X(G^*)$ in $\Esap$. In consequence,
\begin{equation} \label{c2}
\X((H\cPHilz G)^*) \cong \X(H^*) \prod_{\Esap}
\X(G^*)\;\text{in}\;\Pos.
\end{equation}

Therefore, it follows from (\ref{c0}), (\ref{c1}) and (\ref{c2})
that there is an order isomorphism between $\X(H) \prod_{\PHSz}
\X(G)$ and $\X(H) \prod_{\Esap} \X(G)$.
\end{proof}

Let $H$ be a finite algebra of $\PHilz$. Recall that $\KA_H$
denotes the associated base to the $H$-space $\X(H)$. We will
prove that $\X(H^*) = \X(\D(\X(H)))$. In order to show it, first
recall that it was proved in \cite{CCM} that the map
$\epsilon:\X(H) \ra \X(\D(\X(H)))$ given by $\epsilon(x) = \{U\in
\D(X): x\in U\}$ is an order isomorphism. Then the map $\epsilon
\circ \eta: \X(H^*)\ra \X(\D(\X(H)))$ is an order isomorphism too,
where $\eta:\X(H^*) \ra \X(H)$ is the order isomorphism given in
Proposition \ref{finiteH}. Moreover,
\[
\begin{array}
[c]{lllll}
\epsilon(\eta(P)) & = & \{U\in \D(\X(H)): \eta(P) \in U\}  &  & \\
 & = & \{\varphi(a): \varphi^{-1}(P) \in \varphi(a)\} &  &\\
 & = & \{\varphi(a): a\in \varphi^{-1}(P)\}&  &\\
 & = & \{\varphi(a): \varphi(a)\in P\}&  &\\
 & = & P. &  &
\end{array}
\]
Thus, $\epsilon \circ \eta$ is the identity function. Hence,
\[
\X(H^*) = \X(\D(\X(H))).
\]
Since $(\X(\D(\X(H))),\subseteq, \KA_{\D(\X(H))})\in \PHSz$ then
$(\X(H^*),\subseteq, \KA_{\D(\X(H))}) \in \PHSz$. We denote by
$\X(H^*)^{\dag}$ to this object of $\PHSz$.

Therefore, if $H$ is a finite algebra of $\PHilz$ then
$\X(H^*)^{\dag} \in \PHSz$. Let $\eta:\X(H^*) \ra \X(H)$ be the
order isomorphism given in Proposition \ref{finiteH}. Note that
the relation $\eta^{R}\subseteq \X(H^*) \times \X(H)$ is given by
\[
(P,Q) \in \eta^{R}\;\text{if and only if}\;\eta(P) \subseteq Q.
\]

\begin{prop}
Let $H$ and $G$ be finite algebras of $\PHilz$. Then
\begin{enumerate}[\normalfont a)]
\item $U\in\KA_H$ if and only if $\eta^{-1}(U) \in
\KA_{\D(\X(H))}$. \item $\eta^{R}$ is an isomorphism in $\PHSz$.
\item $\X(H) \prod_{\PHSz} \X(G) \cong \X(H^*)^{\dag}
\prod_{\PHSz} \X(G^*)^{\dag}$ in $\PHSz$.
\end{enumerate}
\end{prop}

\begin{proof}
a) Let $U\in \KA_H$, so there exists $a\in H$ such that $U =
\varphi_{H}(a)$. We will prove that
\[
\eta^{-1}(U) = (\varphi_{\D(\X(H))}(\varphi_{H}(a)))^c.
\]
Indeed,
\[
\begin{array}
[c]{lllll}
\eta^{-1}(U) & = & \eta^{-1}(\varphi_{H}(a)^c)  &  & \\
 & = & \{P\in \X(H^*): \eta(P)\in \varphi_{H}(a)^c\} &  &\\
 & = & \{P\in \X(H^*): \varphi_{H}^{-1}(P)\in \varphi_{H}(a)^c\}&  &\\
 & = & \{P\in \X(H^*): a\notin \varphi_{H}^{-1}(P) \}&  &\\
 & = & \{P\in \X(H^*): \varphi_{H}(a)\notin P \} &  &\\
 & = & \{P\in \X(\D(\X(H))): \varphi_{H}(a)\notin P \} &  &\\
 & = & (\varphi_{\D(\X(H))}(\varphi_{H}(a)))^c. &  &
\end{array}
\]
Since $\eta^{-1}(U) = (\varphi_{\D(\X(H))}(\varphi_{H}(a)))^c$
then $\eta^{-1}(U) \in \KA_{\D(\X(H))}$.

Conversely, suppose that $\eta^{-1}(U) \in \KA_{\D(\X(H))}$. Then
there exists $a\in H$ such that $\eta^{-1}(U) =
(\varphi_{\D(\X(H))}(\varphi_{H}(a)))^c$. Our aim is to show that
$U = \varphi_{H}(a)^c$. First note that $P\in \X(H)$ if and only
if there exists $Q_P\in \X(H^*)$ such that $P = \eta(Q_P)$. Thus,
$P\in U$ if and only if $Q_P \in \eta^{-1}(U)$, i.e.,
$\varphi_{H}(a) \notin Q_P$, which is equivalent to $a\notin
\varphi_{H}^{-1}(Q_P)$. But $\varphi_{H}^{-1}(Q_P) = \eta(Q_P) =
P$. Hence, $P\in U$ if and only if $P\notin \varphi_{H}(a)$. Then
$U = \varphi_{H}(a)^c$, so $U\in \KA_H$.

b) It follows from item a) and \cite[Lemma 3.2]{CCM} that
$\eta^{R}$ is an $H$-relation from the $H$-space $\X(H^*)^{\dag}$
to the $H$-space $\X(H)$. Moreover, since $\eta$ is an order
isomorphism it can be proved that $\eta^{R}$ is an $H$-functional
relation, so it is a morphism in $\HS$. Finally, the fact that
$\eta^{R}$ is an isomorphism in $\HS$ follows again from item a)
and \cite[Theorem 3.2]{CCM}. Therefore it follows from
straightforward computations that $\eta^{R}$ is an isomorphism in
$\HS$.
\end{proof}

\begin{lem} \label{finHS}
Let $(X,\tau)$ be a finite topological space and $\KA$ a family of
subsets of $X$.
\begin{enumerate} [\normalfont a)]
\item If $(X,\tau,\KA)$ is an $H$-space then $\tau$ is the set of
downsets of $(X,\pd)$. \item $(X,\tau,\KA)$ is an $H$-space if and
only if $(X,\tau)$ is $T_0$ and $\KA$ is a base of $(X,\tau)$ such
that $\sat(B_1 \cap B_{2}^c) \in \KA$ for every $B_1, B_2 \in
\KA$.
\end{enumerate}
\end{lem}

\begin{proof}
a) We know that there is an order isomorphism and a homemorphism
between $(X,\tau,\KA)$ and $\X(\D(X))$. Then it is enough to prove
that given a Hilbert algebra $H$ we have that $U$ is an open in
$\X(H)$ if and only if $U$ is a downset in $\X(H)$. Let $U$ be a
downset in $\X(H)$. Then $U^{c}\in \X(H)^+$. Thus it follows from
(\ref{finite}) that $U^c$ is a closed set of $\X(H)$, i.e., $U$ is
an open of $\X(H)$. It is immediate the fact that if $U$ is an
open in $\X(H)$ then $U$ is a downset in $\X(H)$.

b) Suppose that $\KA$ is a base of $(X,\tau)$ such that $\sat(B_1
\cap B_{2}^c) \in \KA$ for every $B_1, B_2 \in \KA$. Since $X$ is
finite then the elements of $\KA$ are compact. Hence, $\KA$ is a
base of open compact sets of $(X,\tau)$. Since $X$ is finite and
$(X,\tau)$ is $T_0$ then $(X,\tau)$ is sober.
\end{proof}

Let $(X,\tau_X,\KA_X)$ be a finite $H$-space and $(Y,\leq)$ a
finite poset. Suppose that there exists an order isomorphism
$i:(X,\pd) \ra (Y,\leq)$. Define the family $\KA_Y: = \{i(B):B\in
\KA_X\}$, where $i(B) :=\{i(x):x\in B\}$. Also write $\tau_Y$ for
the family of downsets of $(Y,\leq)$.

\begin{lem} \label{ctp}
The structure $(Y,\tau_Y,\KA_Y)$ is an $H$-space and
$(X,\tau_X,\KA_X) \cong (Y,\tau_Y,\KA_Y)$ in $\HS$. Moreover, if
$(X,\tau_X,\KA_X) \in \PHSz$ then $(Y,\tau_Y,\KA_Y) \in \PHSz$ and
$(X,\tau_X,\KA_X) \cong (Y,\tau_Y,\KA_Y)$ in $\PHSz$.
\end{lem}

\begin{proof}
By a) of Lemma \ref{finHS} we have that $\tau$ is the set of
downsets of $(X,\pd)$. Thus, straightforward computations based in
b) of Lemma \ref{finHS} and the fact that $i:(X,\pd) \ra (Y,\leq)$
is an order isomorphism show that $(Y,\tau_Y,\KA_Y)$ is an
$H$-space.

Now we will prove that $(X,\tau_X,\KA_X) \cong (Y,\tau_Y,\KA_Y)$
in $\HS$. Note that if $B\in \KA_X$ then $i^{-1}(i(B)) = B$.
Hence, $i(B) \in \KA_Y$ if and only if $i^{-1}(i(B)) \in \KA_X$.
By \cite[Lemma 3.2]{CCM} we have that $i^{R}$ is an $H$-relation,
so by \cite[Theorem 3.2]{CCM} we have that $(X,\tau_X,\KA_X) \cong
(Y,\tau_Y,\KA_Y)$ in the category of $H$-spaces and $H$-relations.
Since $i$ is an order isomorphism straightforward computations
prove that $(X,\tau_X,\KA_X) \cong (Y,\tau_Y,\KA_Y)$ in $\HS$. The
rest of the proof follows from straightforward computations.
\end{proof}

\begin{thm}
Let $H, G$ finite algebras in $\PHilz$. Then there exists $\tau$
and $\KA$ families of subsets of $\X(H^*) \prod_{\Esap} \X(G^*)$
such that
\[
P(X,Y): = (\X(H^*) \prod_{\Esap} \X(G^*),\tau,\KA) \in \PHSz
\]
and $\X(H) \prod_{\PHSz} \X(G) \cong P(X,Y)$ in $\PHSz$.
\end{thm}

\begin{proof}
It follows from lemmas \ref{rp1} and \ref{ctp}.
\end{proof}

\section{The product in the category $\Ffin$ of finite forests and open
maps}\label{s5}

A (finite) \emph{forest} is a (finite) poset $F$ such that for
every $x\in F$ the set $(x]$ is a chain. An order preserving map
$f:F\ra G$ between forests is \emph{open} if for every $x\in F$
and $y\in G$, if $y\leq f(x)$ then there exists $z\in X$ such that
$z\leq x$ and $f(z) = y$. Notice that an order preserving map
$f:F\ra G$ between forests is open if for every $x\in X$ it holds
that $f((x]) = (f(x)]$. In this case we also have that if $x$ is
minimal in $F$ then $f(x)$ is minimal in $G$.

We write $\Ffin$ for the category of finite forests and
order-preserving open maps between them. In this section we give
an explicit description of the product of two objects in the
category $\Ffin$.

\begin{defn}
Let $F\in \Ffin$. An $u$-succession $f$ is a set
$\{f_0,\dots,f_n\}$ of elements of $F$ such that $f_0$ is minimal
in $F$ and $f_0< \dots < f_n$. We write $\US(F)$ for the set of
$u$-successions of $F$.
\end{defn}

\begin{rem}
Let $F\in \Ffin$. Let $f_0, \dots,f_n$ elements of $F$ such that
$f_0$ is minimal in $F$ and $f_0\leq f_1\leq \dots \leq f_n$. In
this case we also say that the set $\{f_0,\dots,f_n\}$ is an
$u$-succession because $\{f_0,\dots,f_n\} = \{g_0,\dots,g_m\}$ for
some $g_0,\dots, g_m \in F$ with $g_0 = f_0$ and
$g_0<g_1<\dots<g_m$.
\end{rem}

Let $F\in \Ffin$. Let $f = \{f_0,\dots,f_n\}$ and $g =
\{g_0,\dots,g_m\}$ elements of $\US(F)$. We say that $f\preceq g$
if and only $n\leq m$ and $f_i = g_i$ for every $i=1,\dots,n$. It
is immediate that $\preceq$ is an order. Moreover,
$(\US(F),\preceq) \in \Ffin$. If there is not ambiguity we write
$\US(F)$ in place of $(\US(F),\preceq)$. Note that the order
$\preceq$ defined on $\US(F)$ is exactly the same order defined on
the set of paths of an arbitrary poset. See \cite{AGM} for details
about it.

Let $S, T\in \Ffin$. We define $\pi_1: S\times T \ra S$ by
$\pi_{1}(s,t) = s$ and $\pi_2: S\times T \ra S$ by $\pi_{2}(s,t) =
t$. Let $p\in \US(S\times T)$, i.e., $p = \{p_0,\dots,p_k\}$ where
$p_0,\dots,p_k \in S\times T$, $p_0$ is minimal in $S\times T$ and
$p_0< \dots <p_k$. We also define maps $\hat{\pi_1}: \US(S\times
T) \ra \U(S)$ and  $\hat{\pi_2}: \US(S\times T) \ra \U(T)$ by
$\hat{\pi_1}(p) = \{\pi_{1}(p_0),\dots,\pi_{1}(p_k)\}$ and
$\hat{\pi_2}(p) = \{\pi_{2}(p_0),\dots,\pi_{2}(p_k)\}$. It is
clear that $\hat{\pi_1}$ and $\hat{\pi_2}$ are well defined maps.
Finally we define the set
\[
S\otimes T = \{p\in \US(S\times T): \hat{\pi_{1}}(p) =
(\pi_{1}(p_k)]\;\text{and}\; \hat{\pi_{2}}(p) = (\pi_{2}(p_k)]\}.
\]

\begin{lem}
Let $S,T\in \Ffin$. The maps $\overline{\pi_1}: S\otimes T \ra S$
and $\overline{\pi_2}: S\otimes T \ra T$ given by
$\overline{\pi_1}(p) =$ max $\hat{\pi_{1}}(p)$ and
$\overline{\pi_2}(p) =$ max $\hat{\pi_{2}}(p)$ are morphisms in
$\Ffin$.
\end{lem}

\begin{proof}
It is immediate that $\overline{\pi_1}$ preserves the order. In
order to prove that $\overline{\pi_1}$ is open, let $s\in S$ and
$p\in S\otimes T$ such that $s\leq \overline{\pi_1}(p)$. Since
$p\in S\otimes T$ then $p = \{p_0,\dots,p_k\}$ where
$p_0,\dots,p_k \in S\times T$, $p_0$ is minimal in $S\times T$,
$p_0< \dots <p_k$ and $\hat{\pi_1}(p) = (\pi_{1}(p_k)]$. Thus,
$s\in (\pi_{1}(p_k)]$, so there exists $i=0,\dots,k$ such that $s
= \pi_{1}(p_i)$. Let $q = \{p_0,\dots,p_i\}$. We have that $q\in
S\otimes T$, $q\preceq p$ and $\overline{\pi_1}(q) = s$. Hence,
$\overline{\pi_1}$ is a morphism in $\Ffin$. In a similar way it
can be proved that $\overline{\pi_2}$ is a morphism in $\Ffin$.
\end{proof}

\begin{lem}\label{lemh}
Let $\alpha:F\ra S$ and $\beta: F\ra T$ be morphisms in $\Ffin$.
Then the map $h:F\ra S\otimes T$ given by $h(f) = \{(\alpha(g),
\beta (g)):g\leq f\}$ is a morphism in $\Ffin$ such that
$\overline{\pi_1} \circ h = \alpha$ and $\overline{\pi_2} \circ h
= \beta$.
\end{lem}

\begin{proof}
First we will prove the well definition of $h$. Let $f\in F$.
Since $F$ is finite and $(f]$ is a forest then there exists
$f_0,\dots,f_n\in F$ such that $f_0$ is minimal in $F$ and
$f_0<f_1<\dots<f_n = f$. Thus,
\[
h(f) =
\{(\alpha(f_0),\beta(f_0)),\dots,(\alpha(f_n),\beta(f_n))\}.
\]
Since $\alpha(f_0)$ is minimal in $S$ and $\beta(f_0)$ is minimal
in $T$ then $(\alpha(f_0), \beta(f_0))$ is minimal in $S\times T$.
Moreover, since $\alpha$ and $\beta$ are order-preserving maps
then $(\alpha(f_i),\beta(f_i))\leq (\alpha(f_{i+1}),
\beta(f_{i+1}))$ for every $i = 0,\dots,n$. Hence, $h(f) \in
\US(S\times T)$. Besides, $\hat{\pi_1}(h(f)) =
\{\alpha(f_0),\dots, \alpha(f_n)\} = (\alpha(f_n)]$ because
$\alpha((f_n]) = (\alpha(f_n)]$, so $h(f) \in S\otimes T$. Then
$h$ is a well defined map.

The fact that $h$ preserves the order follows from that $\alpha$
and $\beta$ preserve the order. Now we will show that $h$ is a
morphism in $\Ffin$. Let $f\in F$ and $p\in S\otimes T$ such that
$p\preceq h(f)$. Then there exist $k\leq n$ such that $p =
\{(\alpha(f_0),\beta(f_0)),\dots,(\alpha(f_k),\beta(f_k))\}$.
Hence, $f_k\leq f$ and $h(f_k) = p$. Thus, $h$ is a morphism in
$\Ffin$. It is immediate that $\overline{\pi_1} \circ h = \alpha$
and $\overline{\pi_2} \circ h = \beta$.
\end{proof}

\begin{lem}
Let $\alpha:F\ra S$, $\beta: F\ra T$ and $g:F \ra S\otimes T$
morphisms in $\Ffin$ such that $\overline{\pi_1} \circ g = \alpha$
and $\overline{\pi_2} \circ g = \beta$. Then $g = h$.
\end{lem}

\begin{proof}
Let $f\in F$. Since $F$ is finite and $(f]$ is a forest then there
exists $f_0,\dots,f_n\in F$ such that $f_0$ is minimal in $F$ and
$f_0<f_1<\dots<f_n = f$. Thus,
\[
h(f) =
\{(\alpha(f_0),\beta(f_0)),\dots,(\alpha(f_n),\beta(f_n))\}.
\]
Besides, there exists $(s_0,t_0)<\dots<(s_k,t_k)\in S\times T$
such that $(s_0,t_0)$ is minimal in $S\times T$,
$(s_0,t_0)<(s_1,t_1)<\dots<(s_k,t_k)$ and
\[
g(f) = \{(s_0,t_0),\dots,(s_k,t_k)\}.
\]
We will prove that $g(f) = h(f)$. Consider $(s_i,t_i)$ for some
$i=0,\dots,k$ and $p = \{(s_0,t_0),\dots,(s_k,t_k)\}$. Then $p\in
S\otimes T$ and $p\preceq g(f)$. Thus, there exists $f' \leq f$
such that $g(f') = p$, so there is $j=0,\dots,n$ such that $f' =
f_j$. In consequence we obtain that $\alpha(f_j) =
\overline{\pi_1}(g(f_j)) = s_i$ and $\beta(f_j) =
\overline{\pi_2}(g(f_j)) = t_i$, so $(s_i,t_i) =
(\alpha(f_j),\beta(f_j))\in h(f)$. Then, $g(f)\subseteq h(f)$.
Conversely, consider $(\alpha(f_i),\beta(f_i))$ for some
$i=0,\dots,n$. Since $f_i\leq f$ then $g(f_i)\preceq g(f)$. Thus,
there exists $j=0,\dots,k$ such that $g(f_i) =
\{(s_0,t_0),\dots,(s_j,t_j)\}$. Since $\alpha(f_i) =
\overline{\pi_1}(g(f_i)) = s_j$ and $\beta(f_i) =
\overline{\pi_2}(g(f_i)) = t_j$, we conclude that
$(\alpha(f_i),\beta(f_i)) = (s_j,t_j)\in g(f)$. Then, $h(f)
\subseteq g(f)$. Therefore, $g(f) = h(f)$.
\end{proof}

\begin{prop} \label{prodFfin}
Let $S,T\in \Ffin$. Then $S\otimes T$ is the product between $S$
and $T$ in $\Ffin$.
\end{prop}

\section{The product in $\hFor$}\label{s6}

In this section we use results of Section \ref{s5} in order to
give an explicit description of the product of two objects in
certain category of finite forests $F$ endowed with a
distinguished family of subsets of $F$, which is equivalent to the
category of finite bounded prelinear Hilbert algebras. This
property allow us to obtain an explicit description of the
coproduct of two finite algebras in $\PHilz$. We also give an
explicit description of the coproduct of two finite algebras in
$\PHil$ in terms of the coproduct in $\PHilz$.

We start with the following lemma.

\begin{lem} \label{finHS}
Let $(X,\tau)$ be a finite topological space and $\KA$ a base for
the topology $\tau$ on $X$. Then $(X,\tau,\KA)$ is an $H$-space if
and only if $(X,\leq_{d})$ is a poset, $\tau$ is the set of
downsets of $(X,\leq_{d})$ and $(B_1 \cap B_{2}^c] \in \KA$ for
every $B_1, B_2 \in \KA$.
\end{lem}

\begin{proof}
Let $(X,\tau,\KA)$ be an $H$-space. It follows from Lemma
\ref{finHS} that $(X,\leq_{d})$ is a poset, $\tau$ is the set of
downsets of $(X,\leq_{d})$ and $(B_1 \cap B_{2}^c] \in \KA$ for
every $B_1, B_2 \in \KA$.

Conversely, suppose that $(X,\leq_{d})$ is a poset, $\tau$ is the
set of downsets of $(X,\leq_{d})$ and $(B_1 \cap B_{2}^c] \in \KA$
for every $B_1, B_2 \in \KA$. We will prove that $(X,\tau)$ is
$T_0$. Let $x,y \in X$ such that $x\nleq_{d} y$. Since
$(X,\leq_{d})$ is a poset we can assume that $x\nleq_{d} y$, so
$y\notin \overline{\{x\}}$. Thus, there is $U\in \tau$ such that
$y\in U$ and $U\cap \{x\} = \emptyset$, i.e., $x\notin U$. Hence,
$(X,\tau)$ is $T_0$. Since $X$ is finite then $(X,\tau)$ is sober.
Therefore, $(X,\tau,\KA)$ is an $H$-space.
\end{proof}

\begin{rem} \label{remv}
Let $X$ be a finite set. Then it follows from Lemma \ref{finHS}
that $(X,\tau,\KA)$ is an $H$-space if and only if $(X,\leq)$ is a
poset, $\tau$ is the set of upsets of $(X,\leq)$ and $\KA$ is a
base of $(X,\tau)$ such that $[B_1 \cap B_{2}^c) \in \KA$ for
every $B_1, B_2 \in \KA$.
\end{rem}

If $(X,\leq)$ is a poset and $U\subseteq X$, we write $U^m$ for
the set of minimal elements of $U$.

\begin{defn}
Let $F$ be a finite forest. An $h$-base for $F$ is a family
$\mathcal{B}$ of upsets of $F$ such that satisfies the following
conditions:
\begin{enumerate}[\normalfont 1)]
\item $[x) \in \mathcal{B}$ for every $x\in F$, \item if $B\in
\mathcal{B}$ and $M\subseteq B^m$ then $[M) \in \mathcal{B}$.
\end{enumerate}
\end{defn}

Notice that if $F$ be a finite forest and $\mathcal{B}$ is an
$h$-base of $F$ then $\emptyset \in \mathcal{B}$.

\begin{lem} \label{hb-b}
Let $F$ be a finite forest and $\mathrm{B}$ a family of upsets of
$F$. Then the following conditions are equivalent:
\begin{enumerate}[\normalfont a)]
\item $\mathcal{B}$ is an $h$-base. \item $\mathcal{B}$ is a base
of $F$ with the topology given by the upsets of the poset $F$ such
that $[B_1 \cap B_{2}^c) \in \mathcal{B}$ for every $B_1, B_2 \in
\mathcal{B}$.
\end{enumerate}
\end{lem}

\begin{proof}
Assume the condition a).

First we will prove that $\mathcal{B}$ is a base of $F$. For every
$x\in X$ it is immediate that $x\in [x)$ and $[x)\in \mathcal{B}$.
Let now $x\in B_1 \cap B_2$ with $B_1, B_2 \in \mathcal{B}$. Then
$x\in [x) \subseteq B_1 \cap B_2$ and $[x)\in \mathcal{B}$. Thus,
$\mathcal{B}$ is a base of $F$. We also have that the topology
generated by $\mathcal{B}$ is equal to the set of the upsets of
$F$. In order to show it, let $U$ be in the topology generated by
$\mathcal{B}$. Since the elements of $\mathcal{B}$ are upsets then
$U$ is an upset. Conversely, let $U$ be an upset. Then $U =
\bigcup_{x\in U} [x)$. Since $[x) \in \mathcal{B}$ then $U$ is in
the topology generated by $\mathcal{B}$.

Let $B_1, B_2 \in \mathcal{B}$. In what follows we will see that
$[B_1 \cap B_2^c) \in \mathcal{B}$. Notice that $B_{1}^{m}\cap
B_2^c \subseteq B_{1}^m$. Since $B_1 \in \mathcal{B}$ then
$[B_{1}^{m}\cap B_2^c) \in \mathcal{B}$. Motivated by this fact,
in order to prove that $[B_1 \cap B_2^c) \in \mathcal{B}$ we will
see that $[B_{1}^m \cap B_2^c) = [B_1 \cap B_2^c)$. The inclusion
$[B_{1}^m \cap B_2^c) \subseteq [B_1 \cap B_2^c)$ is immediate.
Conversely, let $x\in [B_1 \cap B_2^c)$. Hence, there exists $y\in
B_1 \cap B_2^c$ such that $y\leq x$. Consider $z\in B_{1}^m$ such
that $z\leq y$. Since $y\notin B_2 = [B_2)$ then $z\notin B_2$, so
$z\in B_{1}^m \cap B_2^c$ and $z\leq x$. Thus, $x\in [B_{1}^m \cap
B_2^c)$. Then $[B_1 \cap B_2^c)\subseteq [B_{1}^m \cap B_2^c)$, so
$[B_{1}^m \cap B_2^c) = [B_1 \cap B_2^c)$, which was our aim.

Finally we have that Remark \ref{remv} and \cite[Lemma 4.1]{CCM}
proves that the condition b) implies the condition a).
\end{proof}

Let $\ChFor$ be the category whose objects are structures
$(X,\leq,\mathcal{B}_X)$ such that $(X,\leq)$ is a finite forest,
$X\in \mathcal{B}_X$ and $\mathcal{B}_X$ is an h-base for
$(X,\leq)$. The morphisms in $\ChFor$ are defined as binary
relations $R:(X,\leq,\mathcal{B}_X) \ra (Y,\leq,\mathcal{B}_Y)$
which satisfy the following conditions:
\begin{enumerate}[\normalfont 1.]
\item $R(x)$ is a non empty downset for every $x\in X$, \item if
$(x,y)\in R$ then there exists $z\in X$ such that $z\leq x$ and
$R(z) = (y]$, \item $R^{-1}(U) \in \mathcal{B}_X$ for every $U\in
\mathcal{B}_Y$.
\end{enumerate}
The identity in $\ChFor$ is the binary relation $\geq$.

Let $\fPHilz$ be the full subcategory of $\PHilz$ whose objects
are finite. In Corollary \ref{corrh} it was proved that there
exists a categorical equivalence between $\PHilz$ and $\PHSz$. The
explicit construction of this equivalence allow us to set that
there exists an equivalence between $\fPHilz$ and the full
subcategory of $\PHSz$ whose objects are also finite. By the
results of this section we have that the last mentioned category
is isomorphic to $\ChFor$. Therefore, there exists a categorical
equivalence between $\fPHilz$ and $\ChFor$.

\begin{rem} \label{remv2}
Let $R:(X,\leq,\mathcal{B}_X) \ra (Y,\leq,\mathcal{B}_Y)$ be a
morphism in $\ChFor$ and $x,y \in X$ such that $x\leq y$. Then
$R(x) \subseteq R(y)$. In order to see it, let $z\in R(x)$, i.e.,
$(x,y)\in R$. Besides $y\geq x$, so $(y,z) \in (\geq \circ R)$. It
follows from Remark \ref{remv} and \cite[Theorem 3.1]{CCM} that
$\geq \circ R = R$, so $(y,z) \in R$, i.e., $z\in R(y)$.
Therefore, $R(x) \subseteq R(y)$.
\end{rem}

\begin{lem}\label{frrf}
Let $(X,\leq,\mathcal{B}_X)$ and $(Y,\leq,\mathcal{B}_Y)$ objects
in the category $\ChFor$.

Suppose that $f:(X,\leq) \ra (Y,\leq)$ is an open map such that
$f^{-1}(U) \in \mathcal{B}_X$ for every $U\in \mathcal{B}_Y$. Then
the binary relation $R(f):(X,\leq,\mathcal{B}_X)\ra
(Y,\leq,\mathcal{B}_Y)$ given by
\begin{center}
$(x,y)\in R(f)$ if and only if $y\leq f(x)$
\end{center}
is a morphism in $\ChFor$.

Conversely, suppose that $R:(X,\leq,\mathcal{B}_X)\ra
(Y,\leq,\mathcal{B}_Y)$ is a morphism in $\ChFor$. Then for every
$x\in X$ there exists the maximum of $R(x)$. Moreover, the map
$f_R:(X,\leq) \ra (Y,\leq)$ is an open map such that
$f_{R}^{-1}(U) \in \mathcal{B}_X$ for every $U\in \mathcal{B}_Y$.
\end{lem}

\begin{proof}
Suppose that $f:(X,\leq) \ra (Y,\leq)$ is an open map such that
$f^{-1}(U) \in \mathcal{B}_X$ for every $U\in \mathcal{B}_Y$. Let
$x\in X$. Then $R(f)(x) = (f(x)]$, so we have that $R(f)(x)$ is a
downset. Since $f(x) \in R(f)(x)$ then $R(f)(x) \neq \emptyset$.
Now consider $(x,y) \in R(f)$, i.e., $y\leq f(x)$. Taking into
account that $f$ is an open map we have that there exists $z\leq
x$ such that $f(z) = y$, so $R(z) = (y]$. Let $U\in
\mathcal{B}_Y$. Straightforward computations show that
$R(f)^{-1}(U) = f^{-1}(U)$. Since $f^{-1}(U) \in \mathcal{B}_X$
then $R(f)^{-1}(U) \in \mathcal{B}_X$.

Conversely, suppose that $R:(X,\leq,\mathcal{B}_X)\ra
(Y,\leq,\mathcal{B}_Y)$ is a morphism in $\ChFor$. Let $x\in X$.
Since $R(x) \neq \emptyset$ and $R(x)$ is finite then the set of
maximal elements of $R(x)$ is non empty. In what follows we will
see that $R(x)$ has a maximum element. Let $y_1, y_2$ maximal
elements in $R(x)$. Since $(x,y_1), (x,y_2) \in R$ then there
exists $z_1, z_2\in X$ such that $z_1 \leq x$, $z_2 \leq x$,
$R(z_1) = (y_1]$ and $R(z_2) = (y_2]$. Taking into account that
$X$ is a forest we deduce that $z_1 \leq z_2$ or $z_2 \leq z_1$.
We can assume that $z_1 \leq z_2$. It follows from Remark
\ref{remv2} that $R(z_1)\subseteq R(z_2)$. Since $y_1 \in R(z_1)
\subseteq R(z_2) = (y_2]$ we have that $y_1 \leq y_2$, so $y_1 =
y_2$. Hence, the set $R(x)$ has a maximum element.

Now we will prove that the map $f_R:(X,\leq) \ra (Y,\leq)$ is an
open map such that $f_{R}^{-1}(U) \in \mathcal{B}_X$ for every
$U\in \mathcal{B}_Y$. Let $x\leq y$. By Remark \ref{remv2} we have
that $R(x) \subseteq R(y)$, so $f_{R}(x) \leq f_{R}(y)$. Hence,
$f_R$ is a monotone map. In order to see that $f_R$ is an open
map, let $y\leq f_{R}(x)$, i.e., $(x,y) \in R$. Then there exists
$z\in X$ such that $z\leq x$ and $R(z) = (y]$. Thus, $f_{R}(z) =
y$. Thus, $f_R$ is an open map. Finally, let $U\in \mathcal{B}_Y$.
Straightforward computations show that $f_{R}^{-1}(U) =
R^{-1}(U)$. Since $R^{-1}(U) \in \mathcal{B}_X$ then
$f_{R}^{-1}(U) \in \mathcal{B}_X$, which was our aim.
\end{proof}

\begin{cor}\label{bijection}
Let $(X,\leq,\mathcal{B}_X)$ and $(Y,\leq,\mathcal{B}_Y)$ be
objects in $\ChFor$. Then there exists a bijection between the set
of open maps from $(X,\leq)$ to $(Y,\leq)$ which satisfy that
$f^{-1}(U) \in \mathcal{B}_X$ for every $U\in \mathcal{B}_Y$ and
the set of morphisms of $\ChFor$ from $(X,\leq,\mathcal{B}_X)$ to
$(Y,\leq,\mathcal{B}_Y)$.
\end{cor}

\begin{proof}
Let $\Gamma$ be the set of open maps from $(X,\leq)$ to $(Y,\leq)$
which satisfy that $f^{-1}(U) \in \mathcal{B}_X$ for every $U\in
\mathcal{B}_Y$ and let $\Delta$ be the set of morphisms of
$\ChFor$ from $(X,\leq,\mathcal{B}_X)$ to
$(Y,\leq,\mathcal{B}_Y)$. Let $F:\Gamma\ra \Delta$ given by $F(f)
= R(f)$. It follows from Lemma \ref{frrf} that it is a well
defined map. The injectivity of $F$ is immediate. In order to
prove that $F$ is a surjective map, let $R\in \Delta$. It follows
from Lemma \ref{frrf} that $f_R \in \Gamma$. We will prove that
$F(f_R) = R$, i.e., $R(f_R) = R$. Let $(x,y) \in R$, i.e., $y\in
R(x)$. In particular, $y\leq f_{R}(x)$. Thus, $y\in R(f_R)(x)$, so
$(x,y) \in R(f_R)$. Then we have proved that $R\subseteq R(f_R)$.
Conversely, let $(x,y) \in R(f_R)$, i.e., $y\leq f_{R}(x)$.
Suppose that $y\notin R(x)$. Since $R(x)^c$ is an upset we have
that $R(x)^c = \bigcup_{i=1}^{n}U_i$, for some $U_1,\ldots,U_n\in
\mathcal{B}_Y$. So there exists $i=1,\ldots,n$ such that $y\in
U_i$. It is immediate that $R(x)\subseteq U_{i}^c$, i.e., that
$x\notin R^{-1}(U_i)$. However, $R^{-1}(U_i) = f_{R}^{-1}(U_i)$,
so $f_{R}(x)\notin U_i$. But $y\in U_i$ and $y\leq f_{R}(x)$, so
since $U_i$ is an upset we deduce that $f_{R}(x)\in U_i$, which is
a contradiction. Then $y\in R(x)$. Thus, $R(f_R) \subseteq R$.
Therefore, $R(f_R) = R$.
\end{proof}

Let $\hFor$ the category whose objects are the objects of $\ChFor$
and whose morphisms are open maps $f:(X,\leq,\mathcal{B}_X) \ra
(Y,\leq,\mathcal{B}_Y)$ such that $f^{-1}(U) \in \mathcal{B}_X$
for every $U\in \mathcal{B}_Y$.

\begin{lem} \label{bijectionb}
Let $(X,\leq,\mathcal{B}_X), (Y,\leq,\mathcal{B}_Y),
(Z,\leq,\mathcal{B}_Z)$ be objects of $\ChFor$. If
$f:(X,\leq,\mathcal{B}_X) \ra (Y,\leq,\mathcal{B}_Y)$ and $g:
(Y,\leq,\mathcal{B}_Y) \ra (Z,\leq,\mathcal{B}_Z)$ are morphisms
in $\hFor$ then $R(g\circ f) = R(g) \circ R(f)$. Conversely, if
$R:(X,\leq,\mathcal{B}_X) \ra (Y,\leq,\mathcal{B}_Y)$ and $S:
(Y,\leq,\mathcal{B}_Y) \ra (Z,\leq,\mathcal{B}_Z)$ are morphisms
in $\hFor$ then $f_{S\circ R} = f_{S} \circ f_{R}$.
\end{lem}

\begin{proof}
Let $f:(X,\leq,\mathcal{B}_X) \ra (Y,\leq,\mathcal{B}_Y)$ and $g:
(Y,\leq,\mathcal{B}_Y) \ra (Z,\leq,\mathcal{B}_Z)$ be morphisms in
$\hFor$. Let $y\in R(g\circ f)(x)$, i.e., $y\leq g(f(x))$. Since
$f(x)\leq f(x)$ we have that $y\in [R(g)\circ R(f)](x)$. Let $y\in
[R(g)\circ R(f)](x)$, i.e., $(x,y) \in R(g) \circ R(f)$. Then
there exists $z$ such that $z\leq g(x)$ and $y\leq f(z)$. Thus,
$y\leq f(z)\leq f(g(x))$, so $y\leq f(g(x))$. Hence, $(x,y) \in
R(g\circ f)$. Then $R(g\circ f) = R(g) \circ R(f)$.

Let $R:(X,\leq,\mathcal{B}_X) \ra (Y,\leq,\mathcal{B}_Y)$ and $S:
(Y,\leq,\mathcal{B}_Y) \ra (Z,\leq,\mathcal{B}_Z)$ be morphisms in
$\hFor$. Let $x\in X$. Then $f_{S\circ R}(x) = \text{max}(S\circ
R)(x)$ and $(f_{S}\circ f_{R})(x) = \text{max}(S(\text{max}\;
R(x)))$. We have that $y\leq \text{max} R(x)$ for every $y\in
R(x)$, so $S(y)\subseteq S(\text{max}\; R(x))$ for every $y\in
R(x)$. Thus, $\text{max}(S(y)) \leq \text{max}(S(\text{max}
R(x)))$ for every $y\in R(x)$, so $\text{max}(S\circ R)(x)\leq
\text{max}(S(\text{max} R(x)))$. On the other hand, $\text{max}
R(x) \in R(x)$, so $S(\text{max} R(x)) \subseteq (S\circ R)(x)$.
Then, $\text{max}(S(\text{max} R(x)))\leq \text{max}(S\circ
R)(x)$. Therefore, $\text{max}(S\circ R)(x)\leq
\text{max}(S(\text{max} R(x)))$.
\end{proof}

Straightforward computations based in Corollary \ref{bijection}
and Lemma \ref{bijectionb} proves the following result.

\begin{prop} \label{fsfo}
The categories $\ChFor$ and $\hFor$ are isomorphic. In particular,
there exists a categorical equivalence between $\hFor$ and
$\fPHilz$.
\end{prop}

The following definition will be useful to obtain a description of
the product in $\hFor$.

\begin{defn}
Let $(X,\leq,\mathcal{B}_X)$ and $(Y,\leq,\mathcal{B}_Y)$ be
objects in $\hFor$. Consider the tern $(X\otimes Y,\preceq,
\mathcal{B}_{X}\otimes \mathcal{B}_Y)$, where
$\mathcal{B}_{X}\otimes \mathcal{B}_Y$ is defined as follows:

$U\in \mathcal{B}_{X}\otimes \mathcal{B}_Y$ if and only if some of
the following conditions are satisfied:
\begin{enumerate}[\normalfont 1.]
\item $U = [u)$ for some $u\in X\otimes Y$. \item $U = [T)$ for
some $T\subseteq [(\overline{\pi_1})^{-1}(V)]^m$ with $V \in
\mathcal{B}_X$. \item $U = [T)$ for some $T\subseteq
[(\overline{\pi_2})^{-1}(V)]^m$ with $V \in \mathcal{B}_Y$.
\end{enumerate}
\end{defn}

Our aim is to prove that $(X\otimes Y, \preceq,
\mathcal{B}_{X}\otimes \mathcal{B}_Y)$ is the product in $\hFor$.
In order to make it possible we will use Proposition
\ref{prodFfin}.

Let $(X,\leq,\mathcal{B}_X), (Y,\leq,\mathcal{B}_Y) \in \hFor$.
Note that by definition $\mathcal{B}_{X}\otimes \mathcal{B}_Y$ is
a family of upsets of $(X\otimes Y, \preceq)$. Besides, $X\in
\mathcal{B}_X$ and $(\overline{\pi_1})^{-1}(X) = X\otimes Y$. Let
$M$ be the set of minimal elements of $X\otimes Y$. Since
$X\otimes Y = [M)$ we have that $X\otimes Y \in
\mathcal{B}_{X}\otimes \mathcal{B}_Y$.

In order to prove that $\mathcal{B}_{X}\otimes \mathcal{B}_Y$ is
an $h$-base of $(X\otimes Y, \preceq)$, let $u\in X\otimes Y$.
Then $[u)\in \mathcal{B}_{X}\otimes \mathcal{B}_Y$. Let $U\in
\mathcal{B}_{X}\otimes \mathcal{B}_Y$ and $M\subseteq U^m$. We
need to prove that $[M) \in \mathcal{B}_{X}\otimes \mathcal{B}_Y$.
If $M = \emptyset$ then $[M) = M \in \mathcal{B}_{X}\otimes
\mathcal{B}_Y$ by definition of $\mathcal{B}_{X}\otimes
\mathcal{B}_Y$. Suppose that $M\neq \emptyset$. If $U = [u)$ for
some $u$ then $[M) = U$, so $[M) \in \mathcal{B}_{X}\otimes
\mathcal{B}_Y$. Assume that $U = [T)$, where $T\subseteq
[(\overline{\pi_1})^{-1}(V)]^m$ and $V \in \mathcal{B}_X$. Since
$M\subseteq U^m\subseteq T$ then $M\subseteq
[(\overline{\pi_1})^{-1}(V)]^m$. Hence, $[M) \in
\mathcal{B}_{X}\otimes \mathcal{B}_Y$. The case  $U = [T)$, where
$T\subseteq [(\overline{\pi_1})^{-1}(V)]^m$ and $V \in
\mathcal{B}_Y$ is similar. So, $\mathcal{B}_{X}\otimes
\mathcal{B}_Y$ is an $h$-base of $X\otimes Y$. Therefore,
$(X\otimes Y, \preceq, \mathcal{B}_{X}\otimes \mathcal{B}_Y)\in
\hFor$.

\begin{lem}
Let $(X,\leq,\mathcal{B}_X)$ and $(Y,\leq,\mathcal{B}_Y)$ be
objects in $\hFor$. Then the maps $\overline{\pi_1}:(X\otimes Y,
\preceq, \mathcal{B}_{X}\otimes \mathcal{B}_Y) \ra
(X,\leq,\mathcal{B}_X)$ and $\overline{\pi_2}:(X\otimes Y,
\preceq, \mathcal{B}_{X}\otimes \mathcal{B}_Y) \ra
(Y,\leq,\mathcal{B}_Y)$ are morphisms in $\hFor$.
\end{lem}

\begin{proof}
We know that $\overline{\pi_1}$ and $\overline{\pi_2}$ are open
maps. Let $U\in \mathcal{B}_X$ and let $M =
[(\overline{\pi_1})^{-1}(U)]^m$. In particular, $M\subseteq
[(\overline{\pi_1})^{-1}(U)]^m$ and $[M) =
(\overline{\pi_1})^{-1}(U)$ because $(\overline{\pi_1})^{-1}(U)$
is an upset. Since $[M) \in \mathcal{B}_X \otimes \mathcal{B}_Y$
we have that $(\overline{\pi_1})^{-1}(U) \in \mathcal{B}_X \otimes
\mathcal{B}_Y$. In the same way can be proved that
$(\overline{\pi_2})^{-1}(U) \in \mathcal{B}_X \otimes
\mathcal{B}_Y$. Thus, $\overline{\pi_1}$ and $\overline{\pi_2}$
are morphisms in $\hFor$.
\end{proof}

Let $(X,\leq,\mathcal{B}_X), (Y,\leq,\mathcal{B}_Y)\in \hFor$. Let
$h: (Z,\leq,\mathcal{B}_Z) \ra (X\otimes Y,\leq,\mathcal{B}_X
\otimes \mathcal{B}_Y)$ be the map given in Lemma \ref{lemh}.
Consider $\alpha:(Z,\leq, \mathcal{B}_Z)\ra
(X,\leq,\mathcal{B}_X)$ and $\beta:(Z,\leq, \mathcal{B}_Z)\ra
(Y,\leq,\mathcal{B}_Y)$ morphisms in $\hFor$ such that
$\overline{\pi_1} \circ h = \alpha$ and $\overline{\pi_2} \circ h
= \beta$.

Let $u\in X\otimes Y$. The following technical lemma will be used
later.

\begin{lem} \label{minu}
Then $[h^{-1}(u)]^m \subseteq
[\alpha^{-1}(\overline{\pi_1}(u))]^m$ or $[h^{-1}(u)]^m \subseteq
[\beta^{-1}(\overline{\pi_2}(u))]^m$.
\end{lem}

\begin{proof}
Let $z, w\in [h^{-1}(u)]^m$. It is enough to prove that $z,w \in
[\alpha^{-1}(\overline{\pi_1}(u))]^m$ or $z,w \in
[\beta^{-1}(\overline{\pi_2}(u))]^m$. Since $z,w \in h^{-1}(u)$
then $h(z) = h(w) = u$, so $\overline{\pi_1}(u) = \alpha(z) =
\alpha(w)$ and $\overline{\pi_2}(u) = \beta(z) = \beta(w)$. Thus,
$z, w \in \alpha^{-1}(\overline{\pi_1}(u))$ and $z,w\in
\beta^{-1}(\overline{\pi_2}(u))$. We will prove that $z,w \in
[\alpha^{-1}(\overline{\pi_1}(u))]^m$ or $z,w \in
[\beta^{-1}(\overline{\pi_2}(u))]^m$.

Let $(z] =\{z_0,z_1,\ldots,z_n\}$ and $(w] =
\{w_0,w_1,\ldots.w_k\}$, where $z_0<z_1<\ldots < z_n = z$ and
$w_0<w_1<\ldots <w_k = w$. By definition of $h$ we have that $h(z)
=\{(\alpha(z_0),\beta(z_0)),\ldots,(\alpha(z_n),\beta(z_n))\}$.
Notice that $(\alpha(z_0),\beta(z_0))\preceq \ldots \preceq
(\alpha(z_{n-1}),\beta(z_{n-1})\preceq
(\alpha(z_n),\beta(z_n))\}$. Suppose $(\alpha(z_{n-1}),
\beta(z_{n-1})) = (\alpha(z_n),\beta(z_n))$. Thus, $h(z_{n-1}) =
h(z_n) = u$. But $z_{n-1}<z_n$, $z_{n-1} \in h^{-1}(u)$ and $z_n
\in [h^{-1}(u)]^m$, which is a contradiction. Hence,
$(\alpha(z_{n-1}), \beta(z_{n-1}))\prec (\alpha(z_n),\beta(z_n))$.
Analogously it can be proved that $(\alpha(w_{k-1}),
\beta(w_{k-1}))\prec (\alpha(w_k),\beta(w_k))$. We write (P) for
the properties $(\alpha(z_{n-1}), \beta(z_{n-1}))\prec
(\alpha(z_n),\beta(z_n))$ and $(\alpha(w_{k-1}),
\beta(w_{k-1}))\prec (\alpha(w_k),\beta(w_k))$. Also note that
$\alpha(z) = \alpha(w)$ and $\beta(z) = \beta(w)$.

Suppose that it is not true that $z,w \in
[\alpha^{-1}(\overline{\pi_1}(u))]^m$. We can assume that $z
\notin [\alpha^{-1}(\overline{\pi_1}(u))]^m$ (the case $w\notin
[\alpha^{-1}(\overline{\pi_1}(u))]^m$ is similar). Let $z^{'}\leq
z$ with $z^{'} \in \beta^{-1}(\overline{\pi_2}(u))$. We will see
that $z = z^{'}$. Since $z \notin
[\alpha^{-1}(\overline{\pi_1}(u))]^m$ then there exists $z^{''}
\in [\alpha^{-1}(\overline{\pi_1}(u))]^m$ such that $z^{''}<z$. In
particular, $\alpha(z) = \alpha(z^{''}$. Then $z^{'}\leq z$ and
$z^{''}\leq z$. Since $Z$ is a forest we have that $z^{'}\leq
z^{''}$ or $z^{''}\leq z^{'}$. Suppose that $z^{'}\leq z^{''}$.
Then $\beta(z) = \beta(z^{'}) \leq \beta(z^{''}) \leq \beta(z)$,
so $\beta(z) = \beta(z^{''})$. Since $\alpha(z) = \alpha(z^{''})$
then $(\alpha(z),\beta(z^{''})$ and $z^{''}<z$ then we obtain a
contradiction by (P). Hence, we have proved that $z^{''}\leq z$.
Then $\alpha(z) = \alpha(z^{''}) \leq \alpha(z^{'}) \leq
\alpha(z)$, so $\alpha(z) = \alpha(z^{'})$. Besides $\beta(z) =
\beta(z^{'})$. Since $(\alpha(z),\beta(z)) =
(\alpha(z^{'},\beta(z^{'})$ and $z^{'}\leq z$ then by (P) we have
that $z = z^{'}$. Thus, $z\in
[\beta^{-1}(\overline{\pi_2}(u))]^m$. Finally we need to prove
that $w\in [\beta^{-1}(\overline{\pi_2}(u))]^m$. If $w\notin
[\alpha^{-1}(\overline{\pi_1}(u))]^m$ then $w\in
[\beta^{-1}(\overline{\pi_2}(u))]^m$, so we can assume that $w\in
[\alpha^{-1}(\overline{\pi_1}(u))]^m$. Suppose that $w\notin
[\beta^{-1}(\overline{\pi_2}(u))]^m$, so there exists $w^{'}$ such
that $w^{'}<w$ and $w^{'}\in \beta^{-1}(\overline{\pi_2}(u))$. In
particular, $\beta(w^{'}) = \beta(w)$. Since $h(z) = h(w)$ and
$w^{'}<w$ then by (P) there exists $z^{'}$ such that $z^{'}<z$ and
$h(z^{'}) = h(w^{'})$. Then $\alpha(z^{'}) = \alpha(w^{'})$ and
$\beta(z^{'}) = \beta(w^{'})$. Taking into account that $w^{'}\in
\beta^{-1}(\overline{\pi_2}(u))$ we have that $\beta(z^{'}) =
\beta(w^{'}) = \beta(w) = \beta(z)$, so $\beta(z) = \beta(z^{'})$.
But $z\in \beta^{-1}(\overline{\pi_2}(u))$, so $z^{'}\in
\beta^{-1}(\overline{\pi_2}(u))$ too. Since $z^{'}<z$ and $z\in
[\beta^{-1}(\overline{\pi_2}(u))]^m$ we have a contradiction.
Therefore, $w\in [\beta^{-1}(\overline{\pi_2}(u))]^m$, which was
our aim.
\end{proof}

\begin{cor}
The map $h$ is a morphism in $\hFor$.
\end{cor}

\begin{proof}
Let $u\in X\otimes Y$. First we will see that $h^{-1}([u)) \in
\mathcal{B}_{X}\otimes \mathcal{B}_Y$. Note that $h^{-1}([u)) =
[h^{-1}(u)) = ([h^{-1}(u)]^m]$. Besides, it follows from Lemma
\ref{minu} that $[h^{-1}(u)]^m \subseteq
[\alpha^{-1}(\overline{\pi_1}(u))]^m$ or $[h^{-1}(u)]^m \subseteq
[\beta^{-1}(\overline{\pi_2}(u))]^m$. Also note that we have
$\alpha^{-1}((\overline{\pi_1})^{-1}(u)) \in B_Z$ and
$\beta^{-1}((\overline{\pi_1})^{-1}(u)) \in \mathcal{B}_Z$ because
$\alpha$ and $\beta$ are morphisms in $\hFor$ and $[u) \in
\mathcal{B}_X \otimes \mathcal{B}_Y$. Thus, by definition of
$h$-base we conclude that $h^{-1}([u)) \in \mathcal{B}_{X}\otimes
\mathcal{B}_Y$.

Let $U = [T)$ for some $T\subseteq X\otimes Y$ and suppose that
$T\subseteq [(\overline{\pi_{1}})^{-1}(V)]^m$ for some $V\in
\mathcal{B}_X$. In particular, $T\subseteq
(\overline{\pi_{1}})^{-1}(V)$. Since $\overline{\pi_1} \circ h =
\alpha$ then $\alpha^{-1}(V) =
h^{-1}[(\overline{\pi_1})^{-1}(V)]$. Thus, $h^{-1}([T)) =
[h^{-1}(T)) \subseteq [\alpha^{-1}(V))$, so $h^{-1}([T))\subseteq
[\alpha^{-1}(V))$. We will prove that $[h^{-1}(T)]^m \subseteq
[\alpha^{-1}(T)]^m$. Let $z\in [h^{-1}(T)]^m$, so $h(z) \in T$.
Then $\overline{\pi_1}(h(z)) = \alpha(z)\in V$, i.e., $z\in
\alpha^{-1}(V)$. Let $w\leq z$ with $w\in \alpha^{-1}(T)$. Since
$\overline{\pi_1}(h(w)) = \alpha(w)$ then $h(w) \in
(\overline{\pi_{1}})^{-1}(V)$. Since $h(z) \in T\subseteq
[(\overline{\pi_{1}})^{-1}(V)]^m$ and $h(w)\leq h(z)$ then $h(z) =
h(w)$. But $z\leq w$ and $z\in [h^{-1}(T)]^m$, so $z = w$. Hence,
$[h^{-1}(T)]^m \subseteq [\alpha^{-1}(T)]^m$. Taking into account
that $\alpha$ is a morphism in $\hFor$ and that $V\in
\mathcal{B}_X$ we have that $\alpha^{-1}(V) \in \mathcal{B}_Z$.
Thus, by definition of $h$-base and the equality $[h^{-1}(T)]^m
\subseteq [\alpha^{-1}(T)]^m$ we have that $[[h^{-1}(T)]^m) \in
\mathcal{B}_Z$.
\end{proof}

The following theorem follows from the previous results of this
section and Proposition \ref{prodFfin}.

\begin{thm}\label{ptp}
Let $(X,\leq,\mathcal{B}_X)$ and $(Y,\leq,\mathcal{B}_Y)$ be
objects in the category $\hFor$. Then $(X\otimes
Y,\preceq,\mathcal{B}_X \otimes \mathcal{B}_Y)$ is the product in
$\hFor$.
\end{thm}

\section{Some remarks on the coproduct of finite prelinear Hilbert
algebras}\label{s7}

In this section we apply the just developed construction of the
product in the category $\hFor$, together with the duality between
$\hFor$ and $\PHilz$, in order to build explicit constructions for
some finite coproducts in both $\PHilz$ and in $\PHil$.
\vspace{5pt}

Recall, from Lemma \ref{lemfin}, that $H \cPHilz G$ is finite
whenever $H$ and  $G$ are finite algebras in $\PHilz$. A
straightforward verification shows that the free prelinear Hilbert
algebra with cero in one generator $p$ coincides with the free
G\"odel algebra in one generator. Hence, its underlying lattice is
as depicted below.

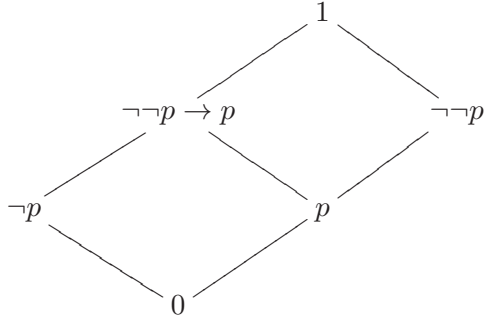
\vspace{5pt}

\centerline{
 \xymatrix{
            &  & 1 \\
            & \neg \neg p \ra p \edge{ur}\edge{dr} & &  \ \ \neg \neg p\ \ \edge{ul}\edge{dl} \\
  \neg p \edge{ur}\edge{dr}  &  & p  \edge{dl}\\
    & 0 & &}
}

In particular, it is a finite algebra and in consequence, any
finitely generated free algebra is also finite. That is to say,
the variety $\PHilz$ is locally finite.

Now, given finite algebras  $H,G$ in $\PHilz$, write $H\cfPHilz G$
for the bounded prelinear Hilbert algebra obtained by dualizing
the construction of Theorem \ref{ptp}, applying Proposition
\ref{fsfo}.

Write $j_H:H\ra H \cfPHilz G$ and $j_G:G\ra H \cfPHilz G$ to the
morphisms induced by the proyections by the duality of Proposition
\ref{fsfo}, and assume there are $f: H\ra K$ and $g:G\ra K$,
morphisms in $\PHilz$. Let $L$ be the subalgebra of $K$ generated
by $f(H) \cup g(G)$, which is finite because $f(H) \cup g(G)$ is
finite and $\PHilz$ is locally finite. Since $L$ is a finite, the
aforementioned duality together with the universal property of the
product in $\hFor$ guarantee the existence of a unique morphism
$h: H \cfPHilz G \ra K$ such that $f = h \circ j_G $ and $g = h
\circ j_H$, proving that $H\cfPHilz G \cong H \cPHilz G$. Hence,
Theorem \ref{ptp} allows us to give and explicit calculation of
the finite coproduct of finite algebras in $\PHilz$. \vskip.3cm

Let us finally see how to adapt the construction for the
coproduct in $\PHilz$, to get an explicit construction for the
coproduct of two finite (non necessarily bounded) prelinear
Hilbert algebras in $\PHil$. Note that prelinear property does
not play any essential role in the following argument, so, if we would
have an explicit construction for the coproduct in the category of
bounded Hilbert algebras, we had been able to get one in the category
of Hilbert algebras too.

Let $(H, \ra, 1) \in \PHil$. Consider the correspondence $H
\mapsto H^0$ that assigns to $H$ the algebra $(H^0, \ra_0 , 1, 0)$
of type (2,0,0), whose universe is the disjoint union $H \sqcup
\{0\}$ and whose binary operation is given by $h \ra_0 k = h \ra
k$, if $h,k \in H$, $0 \ra_0 h = 1$, if $h \in H$ or $h = 0$ and
$h \ra_0 0 = 0$, if $h \in H$. It can be checked that $H^0$, as
defined above, is a bounded prelinear Hilbert algebra and that
this correspondence extends to a functor $(\ )^0 : \PHil \to
\PHilz$ by assigning to $f : H \to K \in \PHil$ the morphism
$$
f^0(h) :=
\left\{
\begin{array}{c}
  f(h), \textrm{ if } h \in H, \\
  0, \ \ \ \ \textrm{ if } h =0 ,
\end{array}
\right.
$$
in $\PHilz$. For any $H \in \PHil$, let us write $i_H: H \to H^0$
for the inclusion of $H$ into $H^0$, which can be seen to be a
morphism in $\PHil$.

Let us now proceed to the construction of the coproduct in $\PHil$
of two finite algebras $G, H \in \PHil$. Let us first compute $H^0
\cPHilz G^0 \in \PHilz$, as before, and write $j_{G^0}$ and
$j_{H^0}$ for the natural inclusions\footnote{It can be seen that
in this case, natural inclusions are in fact injections. Let us
see that $j_{H^0}$ is injective. Take $t: G^0 \to H^0$, the
morphism in $\PHilz$ given by $t(g) = 1$, if $g \in G$ and $t(g) =
0$ if $g = 0$, and consider the cocone in $\PHilz$ $id_{H^0}: H^0
\ra H^0 \leftarrow G^0: t$. By the universal property of the
coproduct, there is a unique $h: H^0 \cPHilz G^0 \to H^0$ in
$\PHilz$, such that $h \circ j_{H^0} = id_{H^0}$. Since,
$id_{H^0}$ is injective, so is $j_{H^0}$.} Put $k_{G} =
j_{G^0}\circ i_G$ and $k_{H} = j_{H^0}\circ i_H$ for the
inclusions of $G$ and $H$ into $G^0 \cPHilz H^0$. Write $G \ast H$
for the Hilbert subalgebra of $G^0 \cPHilz H^0$ generated by
$k_G(G) \cup k_H(H)$, and $j_G = {{k_G}_|}_G$ and $j_H =
{{k_H}_|}_H$, respectively. Let us check that $G \ast H$ with the
injections $k_G$ and $k_H$ is the coproduct of $G$ and $H$ in
$\PHil$. In order to do that, take $K \in \PHil$ endowed with two
morphisms $g: G \to K$ and $h: H \to K$, also in $\PHil$. Since
$g^0: G^0 \to K^0$ and $h^0: H^0 \to K^0$ are in $\PHilz$, by the
universal property of $G^0 \cPHilz H^0$, we have that there is a
unique $k: G^0 \cPHilz H^0 \to K^0$ such that $k \circ j_{G^0} =
g^0$ and $k \circ j_{H^0} = h^0$. Writing $i: G \ast H \to G^0
\cPHilz H^0$ for the inclusion morphism (in $\PHil$), we have that
$(k \circ i)(G \ast H) \subseteq i_K(K)$. Hence, the restriction
of $k$ to  $i(G \ast H)$, is a morphism $\hat{k}: G \ast H \to K
\in \PHil$, that makes the following diagram commute.

$$
\xymatrix{
& K  & \\
G \ar[ur]_{g} \ar[dr]_{j_G} & & H \ar[ul]^{h} \ar[dl]^{j_H}\\
& G\ast H \ar[uu]^{\hat{k}} & }
$$

To see the uniqueness of the morphism $\hat{k}$. Suppose that
$\delta:G\ast H \ra K \in \PHil$ also renders the diagram
commutative. Then $\delta(j_{G}(a)) = g(a) = \hat{k}(j_{G}(a))$
for every $a\in G$ and $\delta(j_{H}(b)) = h(b) =
\hat{k}(j_{H}(b))$ for every $b\in H$. Since $j_{G}(G) \cup
j_{H}(H)$ generates $G\ast H$ then $\delta = \hat{k}$.
\vskip.3cm
Let us end this section calculating a coproduct both in
$\PHilz$ and in $\PHil$ with the procedures described above.

\begin{ex} Let $G = H = 2$, the implicative reduct of the
boolean algebra with two elements. Since 2 is bounded, let
us compute their coproduct in $\PHilz$. Since $\X(2) \cong 1$, and $1 \times 1 = 1$,
we get that $G \cfPHilz H \cong 2$.
\end{ex}

\begin{ex} Let us consider again $G = H = 2$, but now compute their coproduct in $\PHil$.
Let us start by noticing that $G^0 = H^0 = 2^0 = 3$, the implicative reduct of the G\"odel
chain with 3 elements. Hence, $\X(G^0) \cong \X(H^0) \cong \X(3) \cong 2$, here, we write 2 for the
two elements chain. A direct computation of their product in $\hFor$ shows that $\X(G^0 \cfPHilz H^0)$ is
as depicted below.
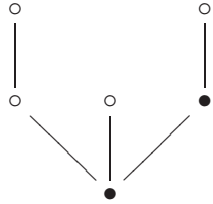
\[
\xymatrix{
\circ              &          & \circ\\
\circ   \ar@{-}[u] & \circ    & \bullet \ar@{-}[u]\\
                   & \bullet \ar@{-}[ur] \ar@{-}[u] \ar@{-}[ul] &
}
\]
Here, the empty bullets correspond to one of the two nontrivial basic sets. The other is symmetrical.
Hence, $G^0 \cfPHilz H^0$ is the Hilbert algebra, generated by $a$ and $b$, whose underlying algebra is depicted bellow.
$$
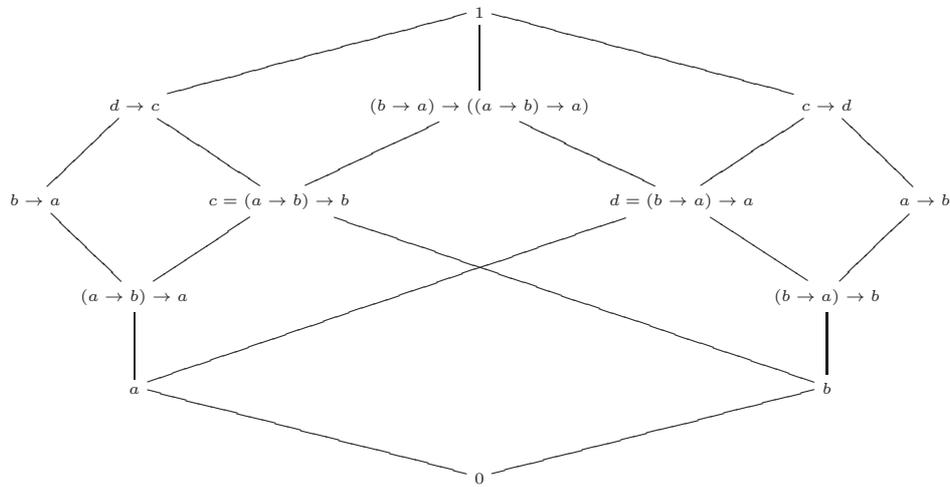
\tiny
{
\xymatrix@C=.2em
{
               &                              &                          &      1                       &                           &                             &              \\
               & d  \to c  \ar@{-}[urr]              &                          &(b \to a) \to ((a \to b) \to a)  \ar@{-}[u]                  &                           & c \to d \ar@{-}[ull]              &               \\
b \to a \ar@{-}[ur]  &                              & c = (a \to b) \to b \ar@{-}[ul] \ar@{-}[ur]&                              & d = (b \to a) \to a \ar@{-}[ul] \ar@{-}[ur] &                             & a \to b \ar@{-}[ul] \\
               & (a \to b) \to a \ar@{-}[ul] \ar@{-}[ur]    &                          &                              &                           & (b \to a) \to b \ar@{-}[ul] \ar@{-}[ur]   &      \\
               & a  \ar@{-}[u] \ar@{-}[uurrr] &                          &                              &                           & b  \ar@{-}[u] \ar@{-}[uulll]&        \\
               &                              &                          & 0 \ar@{-}[urr]  \ar@{-}[ull] &                           &                             &
}
}
$$
Finally, $G \ast H$ is the Hilbert subalgebra $(G^0 \cfPHilz H^0) - \{ 0 \}$, which is the free prelinear Hilbert algebra in two generators.
\end{ex}

\subsection*{Acknowledgments}

This work was supported by CONICET-Argentina [PIP
112-201501-00412].

-----------------------------------------------------------------------------------
\\
\\
Sergio Arturo Celani, \\
Departamento de Matem\'atica,\\
Facultad de Ciencias Exactas (UNCPBA),\\
Pinto 399, Tandil (7000),\\
and Conicet, Argentina,\\
scelani@exa.unicen.edu.ar

-----------------------------------------------------------------------------------
\\
Jos\'e Luis Castiglioni,\\
Departamento de Matem\'atica, \\
Facultad de Ciencias Exactas (UNLP), \\
and CONICET.\\
Casilla de correos 172,\\
La Plata (1900), Argentina.\\
jlc@mate.unlp.edu.ar

-----------------------------------------------------------------------------------------
\\
Hern\'an Javier San Mart\'in,\\
Departamento de Matem\'atica, \\
Facultad de Ciencias Exactas (UNLP), \\
and CONICET.\\
Casilla de correos 172,\\
La Plata (1900), Argentina.\\
hsanmartin@mate.unlp.edu.ar

\end{document}